\documentclass[11pt]{article} 
\usepackage[applemac]{inputenc}
\usepackage[T1]{fontenc}
\usepackage{lmodern}
\usepackage[english]{babel}
\usepackage{float}

\usepackage{amsmath,amsfonts,amssymb,amsthm,mathrsfs,amsrefs, bbm}

\renewcommand{\leq}{\leqslant}
\renewcommand{\geq}{\geqslant}
\usepackage[cmintegrals,libertine]{newtxmath}
\usepackage[cal=euler]{mathalfa}
\usepackage[mono=false]{libertine}
\useosf
\usepackage[a4paper,vmargin={3.5cm,3.5cm},hmargin={2.5cm,2.5cm}]{geometry}
\usepackage[font={small,sf}, labelfont={sf,bf}, margin=1cm]{caption}
\usepackage[pdftex,colorlinks=true]{hyperref}
\usepackage[expansion]{microtype}
\usepackage[pdftex]{color,graphicx,overpic}

\usepackage{stackrel}
\usepackage{soul}

\theoremstyle{plain}
\newtheorem{theorem}{Theorem}
\newtheorem{corollary}[theorem]{Corollary}
\newtheorem{proposition}[theorem]{Proposition}
\newtheorem{lemma}[theorem]{Lemma}
\theoremstyle{definition}

\newtheorem*{question}{Conjecture}
\newtheorem{remark}[theorem]{Remark}

\newcommand{\ind}{\mathbbm{1}}

\newcommand{\PP}{\mathbb{P}}

\newcommand{\eps}{\varepsilon}
\newcommand{\map}{\mathbf{M}}

\hypersetup{
    pdftitle    = {}}

\begin{document}

\title{\bf Universality for random surfaces in unconstrained genus}
\author{\textsc{Thomas Budzinski}\footnote{ENS Paris and Universit\'e Paris-Saclay. E-mail: \href{mailto:thomas.budzinski@ens.fr}{thomas.budzinski@ens.fr}.}, \, \textsc{Nicolas Curien}\footnote{Universit\'e Paris-Saclay and Institut Universitaire de France. E-mail: \href{mailto:nicolas.curien@gmail.com}{nicolas.curien@gmail.com}.} \, and \textsc{Bram Petri}\footnote{Universit\"at Bonn. E-mail: \href{mailto:bpetri@math.uni-bonn.de}{bpetri@math.uni-bonn.de}}}
\date{\today}
\maketitle

\begin{abstract} Starting from an arbitrary sequence of polygons whose total perimeter is $2n$, we can build an (oriented) surface by pairing their sides in a uniform fashion. Chmutov \& Pittel \cite{chmutov2016surface} have shown that, regardless of the configuration of polygons we started with, the degree sequence of the graph obtained this way is remarkably constant in total variation distance and converges towards a Poisson--Dirichlet partition as $n \to \infty$. We actually show that several other geometric properties of the graph are universal. En route we provide an alternative proof of a weak version of the result of Chmutov \& Pittel using probabilistic techniques and related to the circle of ideas around the peeling process of random planar maps. At this occasion we also fill a gap in the existing literature by surveying the properties of a uniform random map with $n$ edges. In particular we show that the diameter of a random map with $n$ edges converges in law towards a random variable taking only values in $\{2,3\}$.
\end{abstract}


\begin{figure}[!h]
 \begin{center}
 \includegraphics[height=4cm]{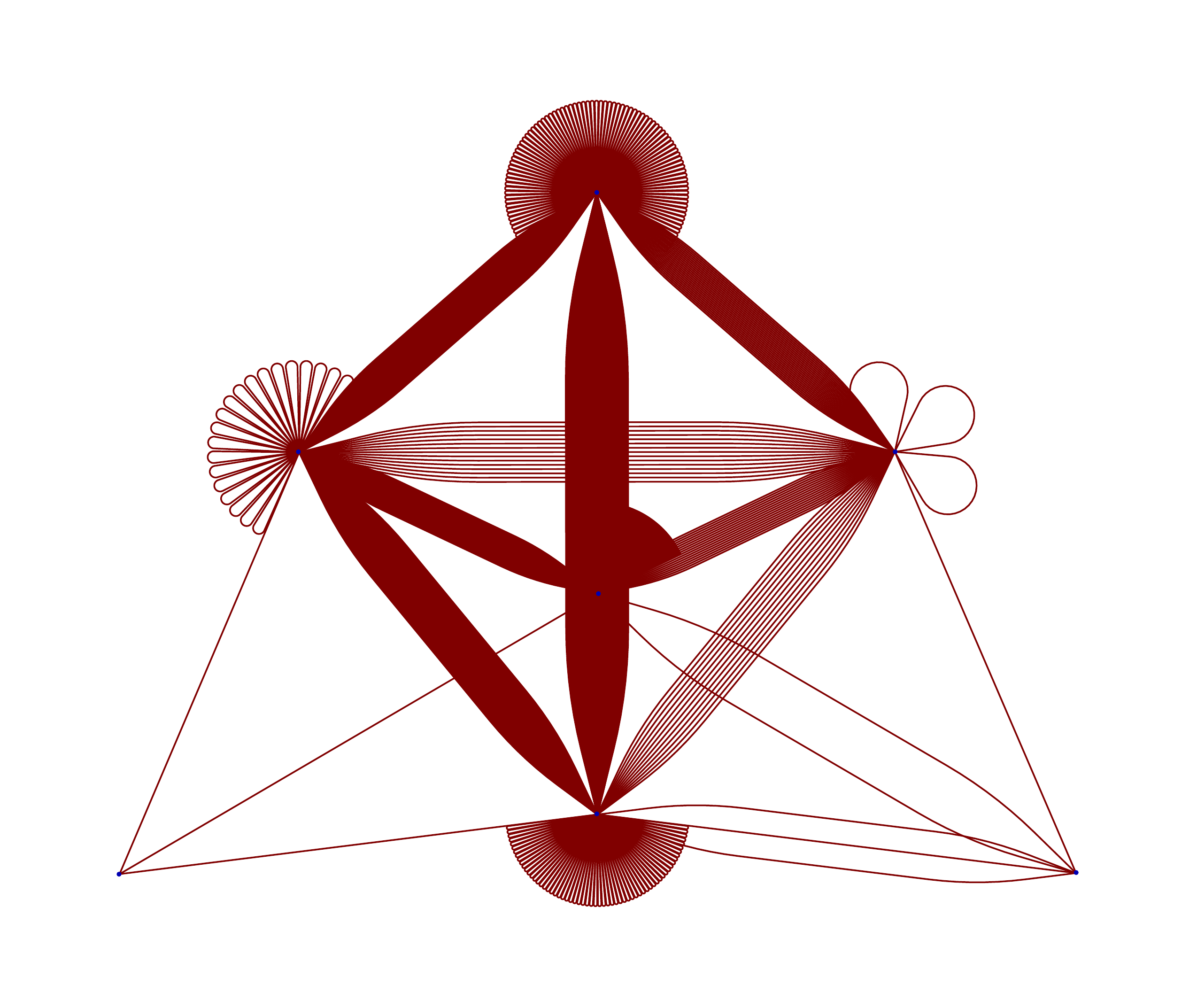}
  \includegraphics[height=4cm]{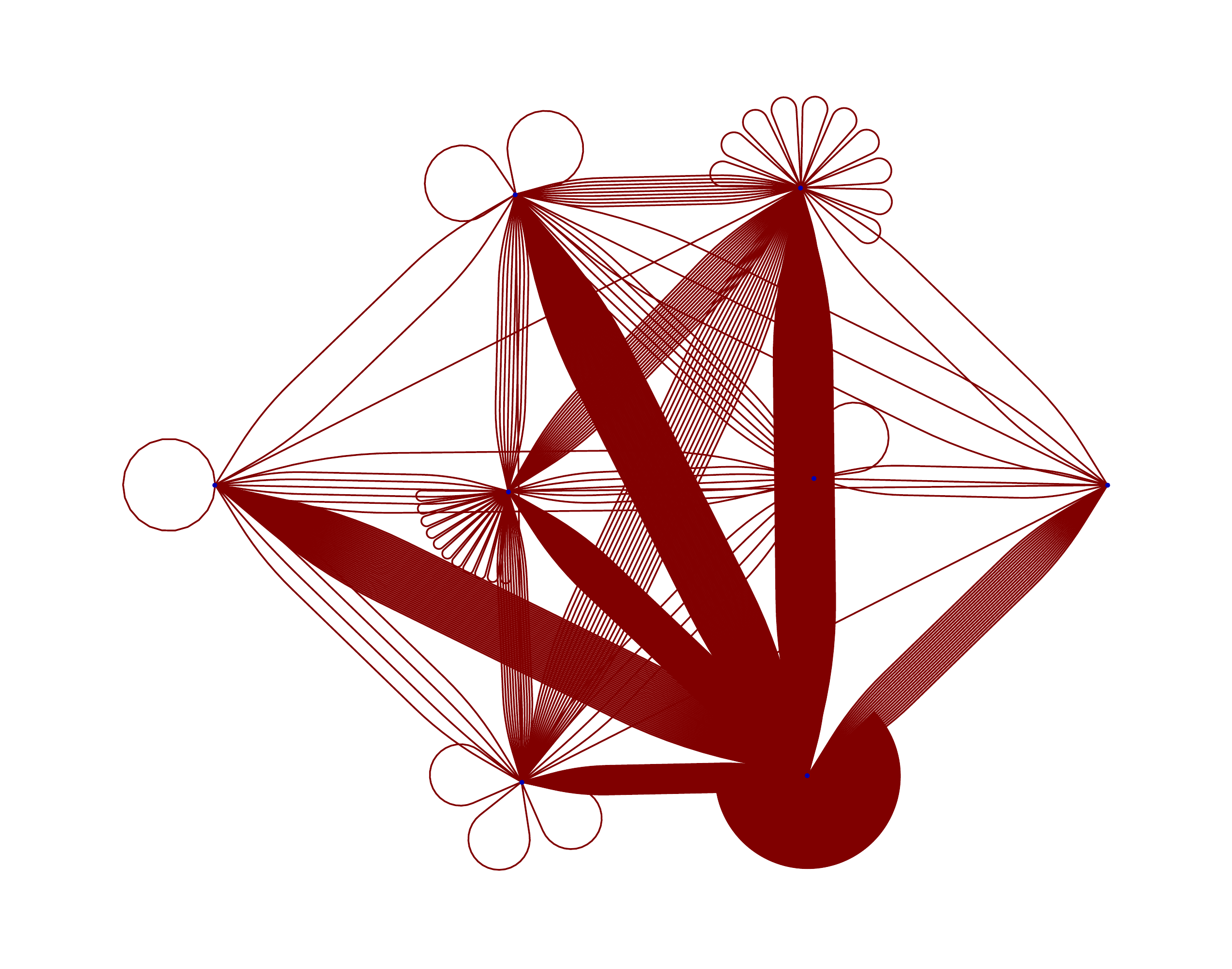}
      \includegraphics[height=4cm]{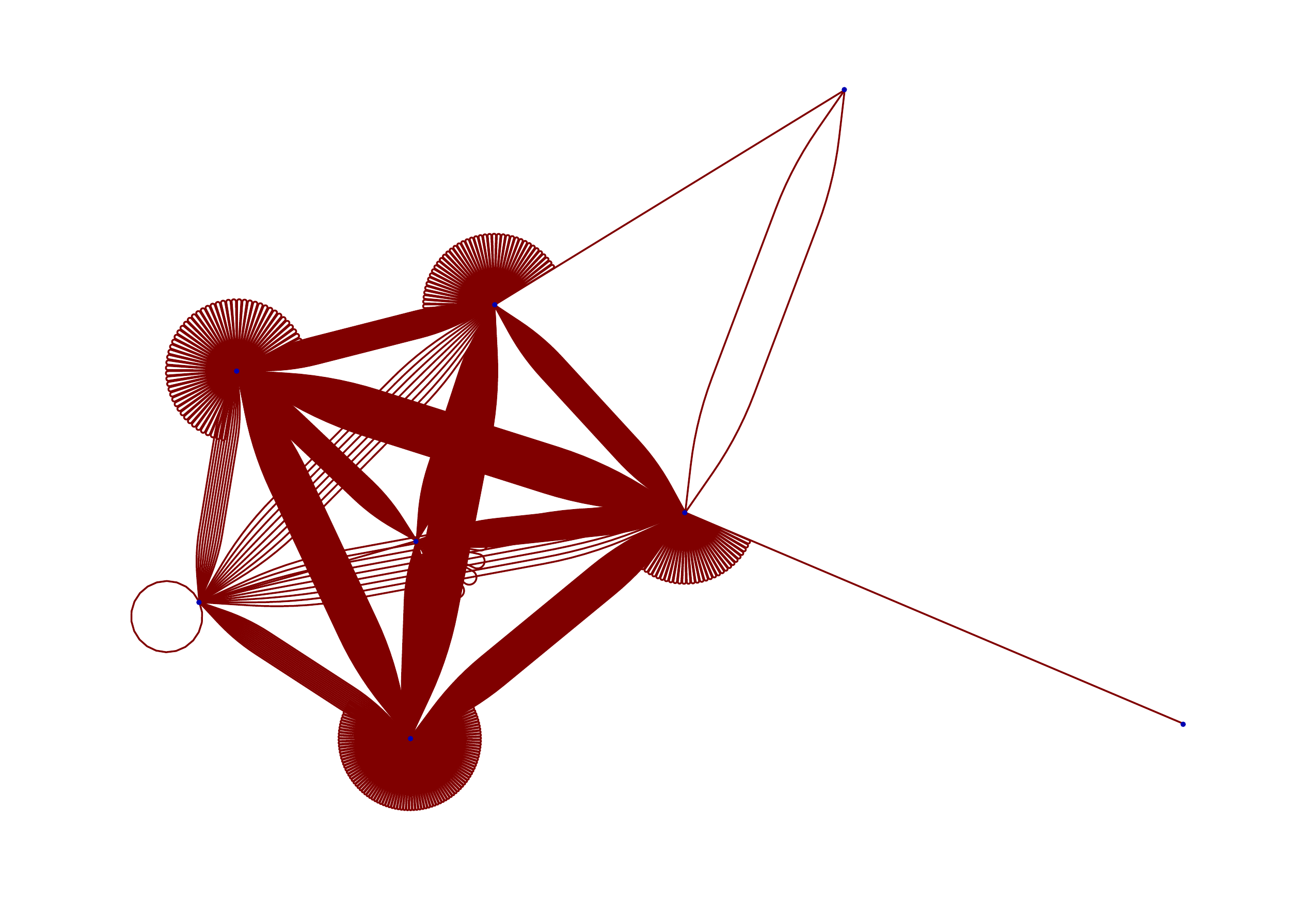}
 \caption{ \label{fig:graphes} Three samples of the graph structure of a uniform random map (genus unconstrained) with $2000$ edges. We see that the graph is highly connected with few vertices carrying many multiple edges and loops. }
 \end{center}
 \end{figure}
\bigskip




\clearpage

\section{Introduction}
\subsection{Gluings of polygons and a conjecture}
Suppose we are given a set of $k\geq 1$ polygons whose perimeters are prescribed by $ \mathcal{P} = \{p_{1}, p_{2}, \dots , p_{k}\}$ where $p_{i} \in \{1,2,3,\dots\}$. We can then form a \emph{random} surface by gluing their sides two-by-two in a uniform manner, see Fig.~\ref{fig:example}.  This model of random surface has been considered e.g.~by Brooks \& Makover \cite{BM04} in the case of the gluing of triangles ($p_{i} = 3$) and later studied by Pippenger \& Schleich \cite{pippenger2006topological} and Chmutov \& Pittel \cite{chmutov2016surface} when $p_{i} \geq 3$. \medskip

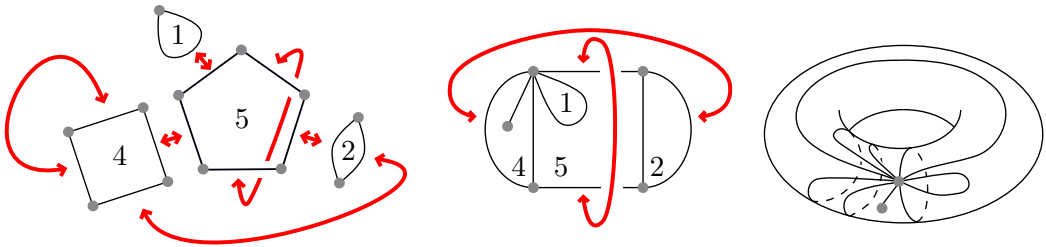
\begin{figure}[!h]
 \begin{center}
  \begin{overpic}[width=14cm]{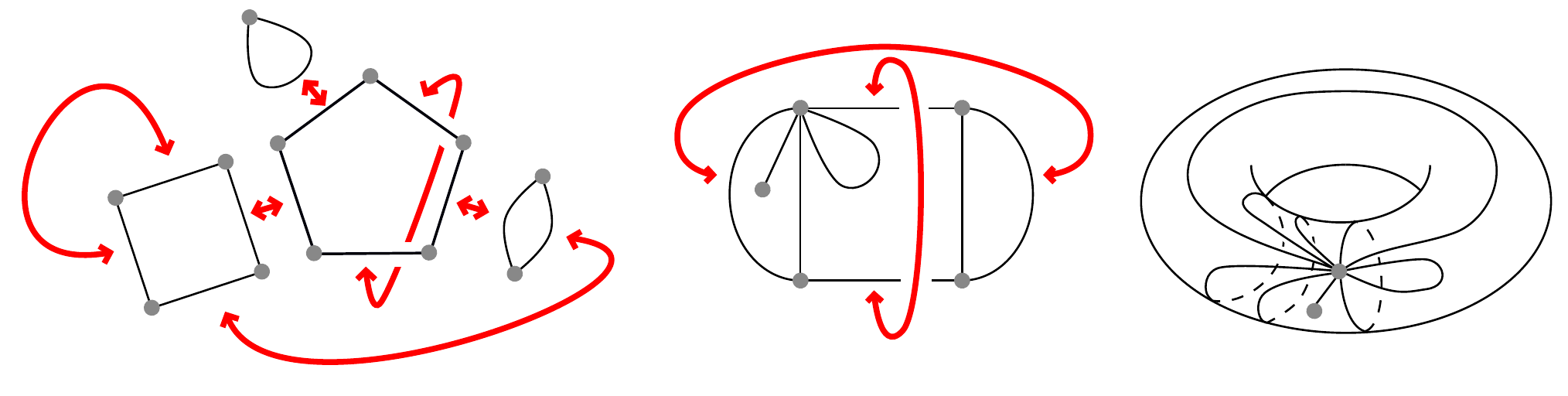}
   \put (17,21) {$1$}
   \put (11.5,9.5) {$4$}
   \put (23,13) {$5$}  
   \put (33,10) {$2$} 
   \put (49,8.5) {$4$}
   \put (53,8.5) {$5$}
   \put (53.5,14.5) {$1$}
   \put (62,8.5) {$2$}   
  \end{overpic} 
 \caption{ \label{fig:example}Creation of a surface by gluing polygons whose perimeters are $1,2,4$ and $5$.}
 \end{center}
 \end{figure}

In this work, we release the constraint that the polygon's perimeters are larger than $3$ but will only consider configurations which do not contain too many $1$ or $2$-gons. If $ \mathcal{P}=\{p_{1}, \dots , p_{k}\}$ is a configuration of polygons, we write 
$$ \# \mathcal{P}=k, \qquad | \mathcal{P}| = \frac{1}{2} \sum_{i=1}^{k} p_{i}, \qquad \mathsf{L}( \mathcal{P}) = \#\{ i : p_{i} =1\}, \qquad \mathsf{B}( \mathcal{P}) = \#\{ i : p_{i} =2\}$$ respectively for half of its total perimeter, the number of $1$-gon (loops) and the number of $2$-gons (bigons) in $ \mathcal{P}$. A sequence  $ (\mathcal{P}_{n})_{n\geq 1}$ of configurations of polygons so that $| \mathcal{P}_{n}|=n$  is said to be \emph{good} if 
 \begin{eqnarray} \label{def:good}  \frac{ \mathsf{L}( \mathcal{P}_{n})}{ \sqrt{n}} \to 0 \quad \mbox{ and }\quad \frac{ \mathsf{B}( \mathcal{P}_{n})}{n} \to 0 \quad \mbox{ as }n \to \infty.  \end{eqnarray}
We focus on good (sequences of) configurations because in this case the surface obtained after the uniform gluing is with high probability connected (see Proposition \ref{prop:connected}).
Specifically, given $ \mathcal{P}$, we randomly label 
the sides of the polygons from $1, 2, \dots, 2| \mathcal{P}|$ in a uniform way and then glue them two by two using an independent pairing of $\{1,2, \dots , 2 | \mathcal{P}|\}$, that is, an involution without fixed points. In all that follows, we only consider oriented surfaces and when we glue two edges we always assume we make sure to preserve the orientation of each polygon\footnote{A heuristic way to formulate this is to imagine that the polygons have two sides, a black side and a white side, and we only glue sides with the same colors when identifying two edges.}. When the gluing is connected, the images of the edges of the polygons form a map $ \mathbf{M}_{ \mathcal{P}}$ with $| \mathcal{P}|$ edges drawn on the  surface created. By abuse of notation we speak of $ \mathbf{M}_{ \mathcal{P}}$ as our ``random surface''. The labeling of the sides of the polygons yields a labeling of the oriented edges of the map by $1,2, \dots , 2 | \mathcal{P}|$: the map $ \mathbf{M}_{ \mathcal{P}}$ is labeled\footnote{Usually a map must be connected and so we put $ \mathbf{M}_{ \mathcal{P}}$ to a cemetery point when the gluing is not connected.}. If we forget the labelings and the orientation of the surface, we get a random multi-graph $\mathbf{G}_{ \mathcal{P}}$, which is the object of study in this work. 

To motivate our results, we start with a conjecture which roughly says that given $|\mathcal{P}|$ and provided the random graph $ \mathbf{G}_{ \mathcal{P}}$ is connected, its law is always the same (in a strong sense), regardless of the configuration of polygons we started with and is close to the law of the random graph obtained from a \emph{uniform random map} with $ | \mathcal{P}|$ edges. However, it is easy to see from Euler's formula that the number of vertices of $ \mathbf{G}_{ \mathcal{P}}$ has the same parity as $| \mathcal{P}|+\# \mathcal{P}$ and so the proper conjecture needs to deal with this parity constraint. Let $ \mathbb{G}_{n}$ be the random graph structure of a uniform random labeled map on $n$ edges, and denote $ \mathbb{G}_{n}^{ \mathrm{odd}}$ (resp. $ \mathbb{G}_{n}^{ \mathrm{even}}$) the random graph $ \mathbb{G}_{n}$ conditioned respectively on having an odd (resp. even) number of vertices.

\begin{question}[Universality for $ \mathbf{G}_{ \mathcal{P}}$] \label{open} Let $ (\mathcal{P}_{n})_{n \geq 1}$ be a good sequence of configurations. We denote by $\epsilon_{n}~\in~\{ \mathrm{even, odd}\}$ the parity of $n + \# \mathcal{P}_{n}$. Then we have
$$ \mathrm{d_{TV}}\left(   \mathbf{G}_{ \mathcal{P}_{n}} ,  \mathbb{G}_{n}^{ \epsilon_{n}}\right) \to 0, \quad \mbox{ as } n \to \infty.$$
\end{question}

Here and later we write $ \mathrm{d_{TV}}(X,Y)$ for the total variation distance between the laws of two random variables $X$ and $Y$. Compelling evidence for the above conjecture is the result of Chmutov \& Pittel \cite{chmutov2016surface} (generalizing work by Gamburd \cite{gamburd2006poisson} in the case when the polygons have the same perimeter), which asserts that  when all the polygon's perimeters are larger than $3$ then up to an error of $O(1/n)$ in total variation distance, the degree distribution of $ \mathbf{G}_{ \mathcal{P}_{n}}$ is the same\footnote{Their result is expressed in terms of the cycle structure of a uniform permutation over $ \mathfrak{A}_{2n}$ if $k$ and $n$ have the same parity (resp.~$ \mathfrak{A}_{2n}^{c}$ if $k$ and $n$ have different parity) where $   \mathfrak{A}_{2n} \subset  \mathfrak{S}_{2n}$ is the group of alternate permutations over $\{1, 2, \dots, 2n\}$. But given our Theorem \ref{thm:maps==PDgluing} this can be rephrased as the degree distribution of $ \mathbb{G}_{n}^{ \mathrm{odd/even}}$.} as that of $ \mathbb{G}_{n}^{ \epsilon_{n}}$.  The proof of \cite{chmutov2016surface} is based on representation theory of the symmetric group. One of the goal of this work is to give a probabilistic proof of a weak version (Theorem \ref{thm:weakPD}) of the above conjecture. We also take this work as a pretext to gather a few results (some of which may belong to the folklore) on the geometry of a uniform random map with $n$ edges.  \medskip 

\hspace{\fill} In the rest of the paper, all the maps considered are labeled. \hspace{\fill}
\medskip 
\subsection{Geometry of random maps}
For $n \geq 1$, we denote by $  \mathbb{M}_{n}$ a random (labeled) map chosen uniformly among all (labeled) maps with $n$ edges. Recall that its underlying graph structure is $ \mathbb{G}_{n}$. It is well known that the distribution of degrees of a random map is closely related to the cycle structure of a uniform permutation, we make this precise in Theorem \ref{thm:maps==PDgluing}.

\paragraph{Permutations and Poisson--Dirichlet distribution.} For $n \geq 0$ we denote by $ \mathcal{U}_{n} = \{u_{1}^{(n)} \geq u_{2}^{{(n)}}  \geq \dots \geq u_{k}^{(n)}\}$ the cycle lengths in decreasing order of a uniform permutation $ \sigma_{n} \in \mathfrak{S}_{n}$. This distribution is well-known and in particular 
 \begin{eqnarray}\label{eq:countpoisson} \mathrm{d_{TV}}\big( \# \mathcal{U}_{n} , \mathrm{Poisson}( \log n)\big) & \xrightarrow[n\to\infty]{}& 0, \\
  \label{eq:cvPD}  \frac{1}{n} \left(u_{i}^{(n)} \right )_{ i \geq1} \xrightarrow[n\to\infty]{(d)}  \mathsf{PD}(1),  \end{eqnarray}
where the (standard) Poisson--Dirichlet distribution $ \mathsf{PD}(1)$ is a probability measure on partitions of $1$, i.e. on sequences $x_{1}>x_{2}>x_{3}>\cdots$ such that $\sum x_{i} =1$, which is obtained by reordering in decreasing order the lengths $U_{1}, U_{2}(1-U_{1}), U_{3}(1-U_{1})(1-U_{2}), \ldots$ where $(U_{i} : i \geq 1)$ is a sequence of i.i.d.~uniform variables on $[0,1]$. We refer e.g.~to \cite{arratia2000limits} for details. 

Recall our notation\footnote{In order to help the reader make the distinction between the concepts, we kept the bb font $( \mathbb{M}, \mathbb{G}, \mathbb{V},\dots)$ for random variables derived from sampling a uniform random map with a fixed number of edges, and the bf font $( \mathbf{M}, \mathbf{G}, \mathbf{V},\dots)$ for random variables associated with the gluing construction started from a configuration of polygons.} $ \map_{ \mathcal{P}}$ for the  map obtained by the uniform labeling and gluing of the sides of oriented polygons whose perimeters are prescribed by $ \mathcal{P}$ 
. If $ \mathcal{P}$ is itself random, one can still consider $ \mathbf{M}_{ \mathcal{P}} $ by first sampling the perimeters in $ \mathcal{P}$ and then performing our random gluing.

\begin{theorem}[Random map as a random gluing with discrete Poisson--Dirichlet perimeters] \label{thm:maps==PDgluing} For some constant $C>0$ we have for all $n \geq 1$ $$ \mathrm{d_{TV}}\left( \mathbb{M}_{n} ; \mathbf{M}_{ \mathcal{U}_{2n}} \right) \leq \frac{C}{n}.$$
In words, a random uniform (labeled) map can be obtained, up to a small error in total variation distance, by a random gluing of random polygons whose sides follow the discrete Poisson--Dirichlet law $ \mathcal{U}_{2n}$.
\end{theorem}

We can deduce several consequences of the above result. First of all, the faces degrees in $  \mathbb{M}_{n}$ have the same law as  $ \mathcal{U}_{2n}$ up to a small error in total variation distance. By well-known results on the distribution of $ \mathcal{U}_{2n}$ (see \cite{arratia2000limits}) this shows that the number of faces of degree $1,2,3,\dots$ in $  \mathbb{M}_{n}$ converge jointly towards independent Poisson random variables of means $1, 1/2, 1/3,\dots$. On the other side, by \eqref{eq:cvPD} the large face degrees, once rescaled by $1/(2n)$, converge towards $ \mathsf{PD}(1)$. Finally by \eqref{eq:countpoisson}, the number of faces $\# \mathsf{F}(  { \map}_{n})$ in $  \mathbb{M}_{n}$ is close in total variation distance to $ \# \mathcal{U}_{2n}$ and thus of $ \mathrm{Poisson}( \log n)$.

 Since $  \mathbb{M}_{n}$ is self-dual, the same results hold for $ { \mathbb{M}}_{n}^{\dagger}$ the dual map of ${ \mathbb{M}}_{n}$. We shall prove in Proposition~\ref{prop:euler} that $\# \mathsf{F}(   \mathbb{M}_{n}), \# \mathsf{V}(  {\mathbb{M}}_{n})$ the number of faces and vertices of $ { \mathbb{M}}_{n}$ and its genus $ \mathsf{Genus}( {\mathbb{M}}_{n})$ obey
  \begin{eqnarray} \label{eq:eulerscaling} \left( \frac{\# \mathsf{V}( \mathbb{M}_{n}) -\log n}{ \sqrt{\log n}} , \frac{\# \mathsf{F}( \mathbb{M}_{n}) -\log n}{ \sqrt{\log n}}, \frac{\mathsf{Genus}( \mathbb{M}_{n}) - \frac{n}{2} + \log n}{ \sqrt{\log n}} \right) \xrightarrow[n\to\infty]{(d)} \left( \mathcal{N}_{1}, \mathcal{N}_{2}, - \frac{\mathcal{N}_{1} + \mathcal{N}_{2}}{2}\right),   \end{eqnarray}
where $ \mathcal{N}_{1}$ and $ \mathcal{N}_{2}$ are independent standard Gaussian random variables. The convergences of the first and second components alone follow from the discussion above. The perhaps surprising phenomenon is that the number of vertices and faces of $ { \mathbb{M}}_{n}$ are asymptotically independent: this is a consequence of our forthcoming Theorem \ref{thm:vertices}. This result was also proved recently by Carrance \cite[Theorem 3.14 and its proof]{carrance2017uniform} by using the techniques of \cite{chmutov2016surface}.

\paragraph{Configuration model.}  We will use Theorem \ref{thm:maps==PDgluing} in conjunction with the \emph{self-duality property} of $  \mathbb{M}_{n}$ to deduce an approximate construction of its graph structure. More precisely, if $ \mathsf{Graph}( \mathfrak{m})$ denotes the (multi)graph obtained from a map $ \mathfrak{m}$ by forgetting the labeling and the cyclic orientation of edges  then, up to an error of $O(1/n)$ in total variation distance, we have
 \begin{eqnarray*} \mathbb{G}_{n}= \mathsf{Graph}( \mathbb{M}_{n} ) \underset{ \mathrm{self\ duality}}{\overset{(d)}{=}} \mathsf{Graph}( \mathbb{M}_{n}^{\dagger} ) \underset{ \mathrm{Thm.\ } \ref{thm:maps==PDgluing}}{\overset{ \mathrm{d_{TV}}}{\approx}}\mathsf{Graph}\left( \mathbf{M}_{ \mathcal{U}_{2n}}^{\dagger} \right).\end{eqnarray*}
The point is that   $\mathsf{Graph}( \mathbf{M}_{ \mathcal{U}_{2n}}^{\dagger} )$ has a particularly simple probabilistic construction: it is obtained as a \emph{configuration model} with degrees prescribed by $ \mathcal{U}_{2n}$. Recall that the configuration model with vertices of degrees $d_{1}, \dots , d_{k}$ is the random graph obtained by starting with $k$ vertices having $d_{1}, \dots, d_{k}$ ``legs'' and pairing those legs two by two in a uniform manner. This model was introduced by Bender \& Canfield \cite{bender1978asymptotic} and Bollob\'as \cite{bollobas1980probabilistic} and was later studied in depth, see e.g.~\cite{RemcoRGII}. This remark is used to prove the following striking property (which is new to the best of our knowledge):

\begin{corollary}[The diameter of a random map is $2$ or $3$.]  \label{cor:diameter}There exists a constant $\xi \in (0,1)$ such that
$$\lim_{n \to \infty} \mathbb{P}(  \mathsf{Diameter}(  \mathbb{M}_{n}) = 3) = 1 -\lim_{n \to \infty} \mathbb{P}(  \mathsf{Diameter}( \mathbb{M}_{n}) = 2) = \xi.$$
\end{corollary}
The proof of Corollary \ref{cor:diameter} gives an expression of $\xi$ in terms of a rather simple random process involving independent Poisson random variables and a $ \mathsf{PD}(1)$ partition. Unfortunately, we have not been able to transform this expression into a close formula. A numerical approximation shows that $\xi \approx 0.3$. 
\medskip 

As mentioned above, the literature concerning the configuration model is abundant. Notice however that the conditions we impose on our perimeters are very different from the usual ``critical'' conditions that can be found e.g.~in \cite{molloy1998size,joseph2014component}. We also mention the works \cite{broutin2014asymptotics,marzouk2018scaling} which study  respectively random plane trees and random \emph{planar} maps with prescribed degrees.

\subsection{Poisson--Dirichlet universality for random surfaces} We now turn to random maps obtained by gluing polygons of prescribed perimeters. We will prove (Proposition \ref{prop:connected}) that the uniform gluing of polygons yields a \emph{connected} surface with high probability provided that we control the number of loops and bigons. 
\paragraph{Number of vertices.} Recall the notation $ \mathbf{M}_{ \mathcal{P}}$ and $  \mathbf{G}_{ \mathcal{P}}$ respectively for the random map and the corresponding graph created by the uniform labeling/gluing of sides of polygons of $ \mathcal{P}$. We also write $\# \mathbf{V}_{ \mathcal{P}}$ for the number of vertices of $\mathbf{M}_{ \mathcal{P}}$.  Finally, for $\alpha >0$, let $ \mathrm{Poisson}^{ \mathrm{odd}}_{\alpha}$ (resp. $\mathrm{Poisson}^{ \mathrm{even}}_{\alpha}$) be a Poisson variable of parameter $\alpha$ conditioned on being odd (resp. even).
 \begin{theorem}[Universality for the number of vertices] \label{thm:vertices} Let $ (\mathcal{P}_{n})_{n \geq 1}$ be a good sequence of configurations and $\epsilon_{n}$ be the parity of $ n+ \# \mathcal{P}_{n}$.  Then we have
 $$ \mathrm{d_{TV}}\left(  \# \mathbf{V}_{ \mathcal{P}_n}, \mathrm{Poisson}^{  \epsilon_{n}}_{\log n}\right) \to 0, \qquad  \mbox{ as } n \to \infty.$$
 \end{theorem}
This is supporting our Conjecture \ref{open}: indeed recall from \eqref{eq:countpoisson} that $ \mathrm{d_{TV}}( \# \mathcal{U}_{2n},  \mathrm{Poisson} (\log n)) \to 0$ as $n \to \infty$. Hence, using our Theorem \ref{thm:maps==PDgluing} and duality, we see that $\mathrm{Poisson}^{ \mathrm{odd}}_{\log n}$ is close in total variation to $\# \mathsf{V}( \mathbb{G}_{n}^{ \mathrm{odd}})$ (and similarly in the even case).  In the case when all the perimeters of  $ \mathcal{P}_{n}$ are larger than $3$, the last result is a trivial consequence of \cite{chmutov2016surface} (although the idea of the proof is very different). In essence, the above theorem says that up to parity considerations, the number of vertices of $ \mathbf{M}_{ \mathcal{P}_{n}}$ is asymptotically independent of its number of faces, which is key in proving \eqref{eq:eulerscaling}. 
 
 \paragraph{Connectivity of the edges.} On the way towards Conjecture \ref{open} we describe the connectivity properties of $  \mathbf{G}_{ \mathcal{P}}$, at least for the vast majority of its edges. We start with a definition. Given a (random) graph $ \mathsf{g}_{n}$ with $n$ edges, denote by $ v_{1},v_{2}, \dots$ the vertices of $\mathsf{g}_{n}$ ordered by decreasing degrees. For $i,j \geq 1$, we write $[{i},{j}]_{\mathsf{g}_{n}}$ for the number of edges between $v_{i}$ and $v_{j}$ with the convention that $[{i},{i}]_{\mathsf{g}_{n}}$ is twice the number of self-loops attached to $v_{i}$. We say that a sequence $(\mathsf{g}_{n})_{n \geq 1}$ of random graphs satisfies the Poisson--Dirichlet universality if 
 \begin{eqnarray}  \left(\frac{[i,j]_{\mathsf{g}_{n}}}{2n} : i,j \geq 1\right)& \xrightarrow[n \to +\infty]{(d)} &\left( X_{i}\cdot X_{j} : i,j \geq 1\right),  \qquad  \mbox{(PDU)} \label{eq:weakPD}\\ & \mbox{ where } & (X_{i})_{i \geq 1} \sim \mathsf{PD}(1). \nonumber
\end{eqnarray}
 \begin{remark} The above Poisson--Dirichlet universality convergence can be rephrased in the theory of edge exchangeable random graphs as the convergence towards the rank 1 multigraph driven by the Poisson--Dirichlet partition, see \cite[Example 7.1 and 7.8]{janson2018edge} and \cite{crane2016edge}. \end{remark}
Using the approximate construction of $ \mathbb{G}_{n}$ as a configuration model based on $ \mathcal{U}_{2n}$, it follows from easy concentration arguments that the graphs $ \mathbb{G}_{n}$ satisfy the Poisson--Dirichlet universality \eqref{eq:weakPD}. On the other hand, an intuitive way to formulate \eqref{eq:weakPD} is that the graphs $\mathsf{g}_{n}$ look like configuration models for a large proportion of the edges. We show that this phenomenon actually holds true for more general polygonal gluings:
 \begin{theorem} \label{thm:weakPD}For any good sequence $ (\mathcal{P}_{n})_{n \geq 1}$ of configurations, the graphs $ \mathbf{G}_{ \mathcal{P}_{n}}$ satisfy the Poisson--Dirichlet universality \eqref{eq:weakPD}.
 \end{theorem}

Let us draw a few consequences of the last theorem. The total number of loops (each loop is counted twice) in the graph $ \mathbf{G}_{ \mathcal{P}_{n}}$ can be written as $\sum_{i} [i,i]_{\mathbf{G}_{ \mathcal{P}_{n}}}$ and hence satisfies
$$ \frac{1}{2n} \# \mathrm{Loops}( \mathbf{G}_{ \mathcal{P}_{n}})  \xrightarrow[n\to\infty]{(d)} \sum_{i \geq 1} X_{i}^{2}.$$ The last convergence (without  identification of the limit law) was recently established in \cite{BCDH18} in the case of uniform maps using the method of moments
. They also studied the degree and the number of edges incident to the root vertex (without counting loops twice) in a random map. Since the root vertex is a degree-biased vertex, if we introduce a random index $I \geq 0$ chosen proportionally to the degree of $v_{I}$ in $ \mathbf{G}_{ \mathcal{P}_{n}}$, then these variables can respectively be written as $$ \sum_{j \geq 0}[I,j]_{  \mathbf{G}_{ \mathcal{P}_{n}}}  \quad  \mbox{ and } \quad  \sum_{j \geq 0}[I,j]_{\mathbf{G}_{ \mathcal{P}_{n}}} - \frac{1}{2} [I,I]_{\mathbf{G}_{ \mathcal{P}_{n}}},$$ and thus converge once rescaled by $1/(2n)$  towards $\overline{X}$ and $\overline{X}-  \frac{1}{2}\overline{X}^{2}$ where $ \overline{X}$ is a size-bias pick   in a Poisson--Dirichlet partition, for which it is well known that $ \overline{X}= \mathrm{Unif}([0,1])$ in distribution. This extends the results of \cite{BCDH18} to a much broader class of random maps.

The above result shows that the distribution of
large degrees is universal among random maps obtained by gluing of good configurations. Actually, the distribution of small degrees, namely the fact that they converge in law towards independent Poisson random variables of means $1, 1/2, 1/3,...$ should also be universal among this class of random maps (and this would indeed be implied by our Conjecture \ref{open}). An approach using the method of moments might be possible but would not fit the general scope of this paper and so we leave this problem for future works.

Finally, let us mention that our proof of Theorem \ref{thm:weakPD} is robust and also allows to obtain results about the location of small faces. For example, if $\mathcal{P}_n$ has a positive proportion of triangles, then the proportion of these triangles whose three vertices are $v_i$, $v_j$ and $v_k$ (in this order) is asymptotically $X_i X_j X_k$ (see Remark \ref{rem:hypergraph}).

\subsection{Organization and techniques} The paper is organized as follows. 

Section \ref{sec:uniformmaps} is devoted to the study of uniform maps with a fixed number of edges and no restriction on the genus. Using the classical \emph{coding of maps by permutations}, we prove that random maps can approximately be seen as a gluing of polygons whose perimeters follow the cycle length of a uniform permutation. By duality, this enables us to see the graph structure of a uniform map as a \emph{configuration model} which is key in our proof of Corollary \ref{cor:diameter}.

We then focus on the more general model $ \mathbf{M}_{ \mathcal{P}}$. We prove Theorem \ref{thm:vertices} and Theorem \ref{thm:weakPD} using ``dynamical'' explorations of the random surfaces $ \mathbf{M}_{ \mathcal{P}}$. Although most of the ideas presented here are already underlying papers in the field (see e.g.~\cite{BM04,pippenger2006topological}), we draw a direct link with the circles of ideas used in the theory of random planar maps and in particular with the peeling process, see \cite{curienpeeling}. We then use two specific algorithms to explore the surface $ \mathbf{M}_{ \mathcal{P}}$ either by peeling the minimal hole or by discovering the vertices one by one, which yield Theorem \ref{thm:vertices} and Theorem \ref{thm:weakPD} respectively.
 
 \medskip\noindent  \textbf{Acknowledgments:} We thank ERC Grant 740943 ``GeoBrown'' and Grant ANR-14-CE25-0014 ``ANR GRAAL'' for support. We are grateful to Julien Courtiel for discussions about \cite{BCDH18}. 
\tableofcontents

\section{Uniform random maps} \label{sec:uniformmaps}

Most of what follows in this section is probably known to many specialists in the field, but we were not able to find precise references and thus took the opportunity to fill a gap in the literature.

\subsection{Uniform maps as a configuration model}
Let $ {\mathfrak{m}}$ be a (connected) labeled map with $n$ edges. The combinatorics of the map is then encoded by two permutations $ \alpha, \phi \in \mathfrak{S}_{2n}$: the permutation $\alpha$ is an involution without fixed points coming from the pairing of the oriented  edges into edges of the map and $\phi$ is the permutation whose cycles are the oriented edges arranged clockwise around each face of the map, see e.g.~\cite[Chapter 1.3.3]{LZ04}. \medskip 

We denote by $ \mathcal{I}_{2n} \subset \mathfrak{S}_{2n}$ the subset of involutions without fixed points (i.e.~product of $n$ non overlapping transpositions). Clearly, we have $\# \mathcal{I}_{2n} = (2n-1)!!$.
Remark that a pair $(\alpha, \phi) \in \mathcal{I}_{2n}\times \mathfrak{S}_{2n}$ is not necessarily associated to a (connected labeled) map: for $n=2$, the "map" associated with the {permutations $\alpha =  (12)(34)$  and $\phi = (1)(2)(3)(4)$ consists} of two disjoint loops. However, this situation is marginal:

\begin{proposition}[Uniform maps are almost uniform permutations]\label{prop:count_connected_maps}
\noindent Let
$ \mathcal{C}_{2n}=\{( \alpha, \phi) \in \mathcal{I}_{2n} \times
\mathfrak{S}_{2n} :  (\alpha, \phi) \mbox{ encodes a  \textbf{connected} labeled map}\}$, so that $ \mathcal{C}_{2n}$ is in bijection with labeled maps with $n$ edges and the number of rooted maps with $n$ edges is $\frac{1}{(2n-1)!}\# \mathcal{C}_{2n}$. Then we have the asymptotic expansion
$$ \frac{\# \mathcal{C}_{2n}}{(2n)! (2n-1)!!} = 1 - \frac{1}{2n} +
O(1/n^{2}).$$
In particular, if $ (A_{n}, F_{n}) \in \mathcal{I}_{2n} \times
\mathfrak{S}_{2n}$ is the pair of permutations associated with a uniform
labeled map with $n$ edges and if $( \alpha_{n}, \phi_{n})$ is uniformly
distributed over $\mathcal{I}_{2n} \times  \mathfrak{S}_{2n}$, then
$$  \mathrm{d_{TV}}\big((A_{n},F_{n}) ; ( \alpha_{n}, \phi_{n}) \big)
\underset{n \to \infty}{\sim} \frac{1}{2n}.$$
\end{proposition}

\noindent \textbf{Proof.} Let $(\alpha_{n}, \phi_{n}) \in \mathcal{I}_{2n}
\times \mathfrak{S}_{2n}$ be uniformly distributed. If $(\alpha_n, \phi_{n})$ does not yield a connected map, that means that the subgroup generated by $\alpha_{n}$ and $\phi_{n}$ does not act transitively on $\{1,2,\dots , 2n\}$, or equivalently that $\{1,2, \dots , 2n\}$ can be partitioned into two non-empty subsets $I$ and $J$ such that both $I$ and $J$ are stable by $\alpha_{n}$ and $\phi_{n}$. By partitioning according to the smallest stable subset containing $1$, we obtain the following recursive relation\footnote{We also note that a very similar (but different) formula appears in \cite{AB00}.}:
$$ \# \mathcal{C}_{2n} =(2n-1)!!(2n)! - \sum_{\ell=0}^{ n-2} {2n-1 \choose 2\ell+1} \#  \mathcal{C}_{2(\ell+1)} (2n-2\ell-3)!!(2n -2 \ell -2)!.$$
Writing $c_{n} = \# \mathcal{C}_{2n}/(2n-1)!$ for the number of rooted maps with $n$ edges and $\xi_{n} = (2n-1)!!$, the above recursive equation is equivalent to $ c_{n} = 2n \times \xi_{n}- \sum_{\ell=1}^{n-1} c_{\ell} \xi_{n-\ell}$. Iterating this  yields
$$ c_{n} = \sum_{i \geq 1} (-1)^{i-1} \sum_{\begin{subarray}{c} k_{1}+ k_{2}+
\dots + k_{i} = n\\ k_{i} \geq 1 \end{subarray}} 2 k_{1} \cdot \xi_{k_{1}}
\xi_{k_{2}}\dots \xi_{k_{i}}.$$
The terms $2 n \xi_{n}$ and $-2(n-1)\xi_{n-1}\xi_{1}$ give
$(2n-1)!!(2n-1+O(1/n))$ while the total sum of the absolute values of the
other terms is easily seen to be of order $O(1/n) \times(2n-1)!!$. This
proves the first claim of the theorem. The second one is a trivial consequence of the definition of total variation distance since $(A_{n}, F_{n})$ is uniformly distributed over $\mathcal{C}_{2n}$. \qed\bigskip

\proof[Proof of Theorem \ref{thm:maps==PDgluing}] Theorem \ref{thm:maps==PDgluing} is an easy consequence of the last proposition. Indeed, if we first sample the polygon's perimeter $ \mathcal{U}_{2n}$ and then assign in a uniform way labels $\{1, 2, \dots , 2n\}$ to the edges of the polygons, this represents a uniform permutation $\phi_{n} \in \mathfrak{S}_{2n} $ whose cycles are the labeled polygons. The independent involution $\alpha_{n}$ without fixed points then plays the role of the gluing operation which identifies the edges of the polygons 2 by 2. Provided the resulting surface is connected, the random labeled map $  {\mathbf{M}}_{ \mathcal{U}_{2n}}$  it creates is plainly associated to the pair of permutation $(\alpha_{n}, \phi_{n}) \in \mathcal{I}_{2n} \times \mathfrak{S}_{2n}$.  \qed  \bigskip

As we alluded to in the introduction, we will use Theorem \ref{thm:maps==PDgluing} in conjunction with the \emph{self-duality property} of $  \mathbb{M}_{n}$ to deduce an approximate construction of its graph structure. Indeed, recall that up to an error of $O(1/n)$ in total variation distance, we have
 \begin{eqnarray} \label{eq:constructionmaps} \mathbb{G}_{n}= \mathsf{Graph}( \mathbb{M}_{n} ) \underset{ \mathrm{self\ duality}}{\overset{(d)}{=}} \mathsf{Graph}( \mathbb{M}_{n}^{\dagger} ) \underset{ \mathrm{Thm.\ } \ref{thm:maps==PDgluing}}{\overset{ \mathrm{d_{TV}}}{\approx}}\mathsf{Graph}( \mathbf{M}_{ \mathcal{U}_{2n}}^{\dagger} ) \ \ {\overset{(d)}{=}} \ \   \mathsf{ConfigModel}( \mathcal{U}_{2n}), \end{eqnarray}
 where $ \mathsf{ConfigModel}(\{d_{1}, d_{2}, \dots , d_{k}\})$ is the random graph obtained by starting with $k$ vertices having $d_{1}, \dots , d_{k}$ half-legs and labeling/pairing those $2$ by $2$ in a uniform manner, see \cite{bollobas1980probabilistic}.

\subsection{Euler's relation in the limit}
We prove \eqref{eq:eulerscaling} which we recall below.
\begin{proposition}[Euler's relation in the limit] \label{prop:euler} If $\# \mathsf{F}(   \mathbb{M}_{n}), \# \mathsf{V}(  {\mathbb{M}}_{n})$ and $ \mathsf{Genus}( {\mathbb{M}}_{n})$ are respectively the number of faces, vertices and the genus of $\mathbb{M}_{n}$ then we have 
  \begin{eqnarray*} \left( \frac{\# \mathsf{V}( \mathbb{M}_{n}) -\log n}{ \sqrt{\log n}} , \frac{\# \mathsf{F}( \mathbb{M}_{n}) -\log n}{ \sqrt{\log n}}, \frac{\mathsf{Genus}( \mathbb{M}_{n}) - \frac{n}{2} + \log n}{ \sqrt{\log n}} \right) \xrightarrow[n\to\infty]{(d)} \left( \mathcal{N}_{1}, \mathcal{N}_{2}, - \frac{\mathcal{N}_{1} + \mathcal{N}_{2}}{2}\right),   \end{eqnarray*}
where $ \mathcal{N}_{1}$ and $ \mathcal{N}_{2}$ are independent standard Gaussian random variables. \end{proposition}
\proof We will rely on Theorem \ref{thm:vertices} which is proved later in the paper. We beg the reader's pardon for this inelegant ``back to the future'' construction.  Let us first focus on the second component. By Theorem \ref{thm:maps==PDgluing} together with  \eqref{eq:countpoisson} we deduce that 
$$\# \mathsf{F}( \mathbb{M}_{n})  \overset{ \mathrm{d_{TV}}}{\approx} \# \mathcal{U}_{2n} \overset{ \mathrm{d_{TV}}}{\approx} \mathrm{Poisson} ( \log 2n) \overset{ \mathrm{d_{TV}}}{\approx} \mathrm{Poisson}( \log n),$$ since for any $c>0$ we have $  \mathrm{d_{TV}}(\mathrm{Poisson} (\lambda +c) , \mathrm{Poisson}(\lambda)) \to 0$ as $\lambda \to \infty$. It is then standard that $  \lambda^{{-1/2}}(\mathrm{Poisson}(\lambda)-\lambda)$ converges in law towards a standard Gaussian as $\lambda \to \infty$. This proves the convergence of the second coordinate, the first one being obtained by self-duality. It remains to show that the (rescaled) number of faces and vertices are asymptotically independent. Once this is done, the convergence of the third component follows from Euler's relation: $$ \# \mathsf{V}( {\mathbb{M}_{n}}) +\# \mathsf{F}( { \mathbb{M}_{n}})-n = 2(1- \mathsf{Genus}( { \mathbb{M}_{n}})).$$
To prove the asymptotic independence, we rely on Theorem \ref{thm:vertices}.  Indeed, if $( \mathcal{U}_{2n})_{n \geq 1}$ is the cycle structure of a uniform random permutation of $ \mathfrak{S}_{2n}$, then as recalled in the introduction, the numbers of loops and bigons in $ \mathcal{U}_{2n}$ satisfy
$$ ( \mathsf{L}( \mathcal{U}_{2n}), \mathsf{B}( \mathcal{U}_{2n})) \xrightarrow[n\to\infty]{(d)} \left( \mathrm{Poisson}(1), \mathrm{Poisson}\left(\frac{1}{2}\right)\right),$$ with independent Poisson variables. In particular, for any $ \varepsilon>0$, the probability that either $ \mathsf{B}( \mathcal{U}_{2n}) \geq n^{ 1 - \varepsilon}$ or $\mathsf{L}( \mathcal{U}_{2n}) \geq n^{\frac{1}{2}- \varepsilon}$ tends to $0$ as $ n \to \infty$. We can then apply Theorem \ref{thm:vertices} (using its notation) to deduce that 
$$ \mathbb{E}\left[ \mathrm{d_{TV}} \left( \# \mathsf{V}( \mathbf{M}_{ \mathcal{U}_{2n}}), \mathrm{Poisson}^{ \epsilon_{n}}_{\log n} \right) ~\Big | ~ \mathcal{U}_{2n},\mathsf{B}( \mathcal{U}_{2n}) \leq n^{ 1- \varepsilon}, \mathsf{L}( \mathcal{U}_{2n}) \leq n^{\frac{1}{2}- \varepsilon}  \right] \xrightarrow[n \to \infty]{\mathbb{P}} 0,$$
where $ \epsilon_{n}$ is the parity of $n + \# \mathcal{U}_{2n}$. But since $\lambda^{{-1/2}}(\mathrm{Poisson}^{\epsilon}_{\lambda}-\lambda) \to \mathcal{N}$ regardless of the parity $ \epsilon$ we indeed deduce that the rescaled number of faces and vertices in $ \mathbf{M}_{ \mathcal{U}_{2n}}$ are asymptotically independent. The same is true in $ \mathbb{M}_{n}$ by Theorem \ref{thm:maps==PDgluing}. \endproof

\subsection{The diameter is $2$ or $3$}
In this section we prove Corollary \ref{cor:diameter}, showing that the diameter of a random map with $n$ edges converges in law towards a random variable whose support is $\{2,3\}$. Before going into the proof, let us sketch the main idea (a quick glance at Figure \ref{fig:graphes} may help to get convinced). We first prove that with high probability all vertices have a neighbor among the vertices of degree comparable to $n$. Since all these vertices are linked to each other, this proves that the diameter of the map is at most $3$. The diameter of the map can even be equal to $2$ if all pairs of vertices of low degree share a neighbor among the vertices of high degree.
\medskip 

Let us recall a simple calculation that we will use several times below. If $  \mathcal{U}_{2n}$ is the cycle structure of a uniform random permutation of $ \mathfrak{S}_{2n}$, then for any positive function $\psi : \{1,2, \dots \} \to \mathbb{R}_{+}$ we have 
 \begin{eqnarray} \label{eq:sizebias} \mathbb{E}\left[\sum_{d \in \mathcal{U}_{2n}} \psi(d)\right] =  2n \cdot  \mathbb{E}\left[\frac{1}{2n} \sum_{d \in \mathcal{U}_{2n}} d \cdot \frac{\psi(d)}{d} \right] = 2n \cdot  \frac{1}{2n}\sum_{i=1}^{2n} \frac{\psi(i)}{i}= \sum_{i=1}^{2n} \frac{\psi(i)}{i},  \end{eqnarray}
 simply because the law of the length of a typical cycle in a uniform permutation of $ \mathfrak{S}_{2n}$ (i.e.~a cycle sampled proportionally to its length or equivalently the cycle containing a given point) is uniformly distributed over $\{1,2, \dots , 2n\}$.

\proof[Proof of Corollary \ref{cor:diameter}] According to \eqref{eq:constructionmaps}, it suffices to prove the result for the random graph obtained from a configuration model whose vertices have degrees prescribed by $ \mathcal{U}_{2n}$, which we denote below by $ \mathsf{ConfigModel} (\mathcal{U}_{2n})$. We write $C_{1}^{(n)}, \dots, C_{2n}^{(n)}$ for the number of occurrences of  $1, 2, \dots, 2n$ in $ \mathcal{U}_{2n}$, i.e.~the number of vertices of degree $1,2, \dots, 2n$ in $ \mathsf{ConfigModel}( \mathcal{U}_{2n})$. For $ \delta >0$ we denote by $ \mathrm{V}^{(n)}_{\delta n}$ the set of vertices of $\mathsf{ConfigModel}( \mathcal{U}_{2n})$ whose degree is larger than or equal to $\delta n$. We first state two lemmas:
\begin{lemma} \label{lem:core} For any $ \varepsilon>0$, we can find $\delta>0$ and $A>0$ such that for all $n$ sufficiently large, in the random graph $\mathsf{ConfigModel}( \mathcal{U}_{2n})$, we have 
 \begin{eqnarray} \label{eq:connectedgrand} \mathbb{P}\left( \mbox{all vertices have a neighbor in }  \mathrm{V}^{(n)}_{ \delta n}\right) \geq 1- \varepsilon,  \\
 \mathbb{P}\left(\mbox{all vertices with degree  } \geq A \mbox{ are connected to all the vertices of } \mathrm{V}^{(n)}_{ \delta n}\right) \geq 1- \varepsilon. \label{eq:pairgrand}  \end{eqnarray}
\end{lemma}

The next lemma is a decoupling result between the small degrees and the large degrees in $ \mathcal{U}_{2n}$:\begin{lemma} \label{lem:decoupling} If $C_{1}^{(n)}, C_{2}^{(n)}, \dots$ denote the number of occurences of $1,2,\dots$ in $  \mathcal{U}_{2n}$ and if $M_{1}^{(n)}>M_{2}^{(n)}> \dots$  are the degrees of $ \mathcal{U}_{2n}$ ranked in decreasing order, then we have the following convergence in distribution in the sense of finite-dimensional marginals:
 \begin{eqnarray*} \left( \left(C_{i}^{(n)}\right)_{i \geq 1} , \left(\frac{M_{i}^{(n)}}{2n}\right)_{i \geq 1}\right) & \xrightarrow[n\to\infty]{(d)} & ((P_{i})_{i \geq 1} , \mathsf{PD}(1)),  \end{eqnarray*}
 where $P_{1},P_{2}, \dots$ are independent Poisson random variables of mean $ 1, \frac{1}{2}, \frac{1}{3}, \dots$ and $ \mathsf{PD}(1)$ is an independent Poisson--Dirichlet partition. 
 \end{lemma}

\begin{proof}[Proof of Corollary \ref{cor:diameter} given Lemmas \ref{lem:core} and \ref{lem:decoupling}]
Given the last two lemmas, the proof of Corollary \ref{cor:diameter} is rather straightforward. Indeed, combining \eqref{eq:connectedgrand} and \eqref{eq:pairgrand} we deduce that with probability at least $1- 2 \varepsilon$ all vertices of $ \mathsf{ConfigModel}( \mathcal{U}_{2n})$ are linked to $ \mathrm{V}^{(n)}_{\delta n}$, and if $ \delta n \geq A$ then all the vertices of $ \mathrm{V}^{(n)}_{\delta n}$ are linked to each other. Hence the diameter of the graph is less than $3$ and we even have $ \mathrm{d_{gr}}(u,v) \leq 2$ as soon as $u$ or $v$ has degree larger than $A$. On the other hand, we clearly have $ \mathrm{Diameter}( \mathsf{ConfigModel}( \mathcal{U}_{2n})) \geq 2$ with high probability. To see this, just consider a vertex $u$ with degree of order $O(1)$: this vertex cannot be linked to all the vertices (there are $\approx \log n$ vertices) and thus must be at distance at least $2$ from another vertex in $\mathsf{ConfigModel}( \mathcal{U}_{2n})$. By this reasoning, up to an event of probability $2 \varepsilon$, to decide whether the diameter of $\mathsf{ConfigModel}( \mathcal{U}_{2n})$ is $2$ or $3$, one must know whether or not we can find two vertices $ u,v$ such that 
  \begin{eqnarray}\label{eq:event}\begin{array}{l}\mathrm{deg}(u) \leq A,\\ \mathrm{deg}(v) \leq A, \end{array} \quad \mbox{ and } u,v \mbox{ do not share a neighbor in } \mathrm{V}^{(n)}_{\delta n}.  \end{eqnarray}
But clearly, by Lemma \ref{lem:decoupling} and the definition of the configuration model, one can describe the limit in distribution of the connections of all vertices of low degree as follows. Consider $P_{1},P_{2}, \dots$ independent Poisson random variables of means $ 1, \frac{1}{2}, \frac{1}{3}, \dots$. When $P_{i}>0$, we imagine that we have $P_{i}$ vertices $v_{1}^{(i)}, \dots, v_{P_{i}}^{(i)}$ of ``degree $i$'', each of them carrying $i$ independent uniform random variables over $[0,1]$ denoted by $$U_{1}(v_{k}^{(i)}), \dots , U_{i}(v_{k}^{(i)}), \qquad \forall i \geq 1, \forall 1 \leq k \leq P_{i}.$$ Independently of this, consider $ \mathbf{X} = (X_{1}>X_{2}>\cdots)$ a partition of unity distributed according to $ \mathsf{PD}(1)$. The vertices $v_{k}^{(i)}$ describe the low degree vertices, $X_{a}$ describe the large degree vertices and $U_{\ell}(v_{k}^{(i)})$ describe the connections of those low degree vertices to the large degree vertices. Hence, for two distinct ``small vertices'' $v^{(i)}_{k}$ and $v_{\ell}^{(j)}$, we say that they share a neighbor if we can find $1 \leq a \leq i$ and $1 \leq b \leq j$ such that $U_{a}(v^{(i)}_{k})$ and $U_{b}(v^{(j)}_{\ell})$ fall into the same component of $[0,1]$ induced by $ \mathbf{X}$. If we put 
$$ \xi = \mathbb{P}( \mbox{there exist two vertices which do not share a neighbor}),$$
then we leave the reader verify that $\xi$ is the limit of the probability of the event in \eqref{eq:event}. But since the event in \eqref{eq:event} is, up to an error of probability at most $2 \varepsilon$, the same event as $\{\mathrm{Diameter}({ \mathbb{M}}_{n}) = 3 \}$ when $n$ is large, the corollary is proved. Finally, elementary computations\footnote{See also the proof of Lemma \ref{lem:core} below.} show that in the limit model, almost surely, the number of pairs of small vertices with no common neighbour is finite and that $0 < \xi <1$.
\end{proof} 

We now prove the two lemmas. \medskip 
 
\begin{proof}[Proof of Lemma \ref{lem:core}]
We start with the first point. Let $ \delta>0$ and let us write 
$$ \sum_{ u \in  \mathrm{V}^{(n)}_{\delta n}} \mathrm{deg}(u)  = 2n (1-   E(\delta,n)).$$
By the convergence of the renormalized degrees towards the Poisson--Dirichlet partition (actually, we only use the fact that this is a partition of unity), for any $ \varepsilon \in (0, 1/2)$ we can find $\delta >0$ such that for all $n$ large enough we have 
$$  \mathbb{P}(E(\delta,n) \geq \varepsilon) \leq \varepsilon.$$
By the definition of the configuration model, conditionally on $ \mathcal{U}_{2n}$, the probability that a given vertex  of degree $d \geq 1$ is not linked to $ \mathrm{V}^{(n)}_{\delta n}$ is bounded above by $E(n, \delta) ^{d}$. Hence we deduce that the probability of the event in \eqref{eq:connectedgrand} is bounded as follows:   \begin{eqnarray*} 1-\mathbb{P}\left( \mbox{all vertices have a neighbor in }  \mathrm{V}^{(n)}_{ \delta n}\right) & \leq & \mathbb{E}\left[ \sum_{ d \in \mathcal{U}_{2n}} E(n, \delta) ^{d} \right ] \\
 & \leq & \mathbb{P}( E(n, \delta) \geq  \varepsilon) + \mathbb{E}\left[ \sum_{d \in \mathcal{U}_{2n}}  \varepsilon^{d}\right]\\ & \underset{ \eqref{eq:sizebias}}{=} &\mathbb{P}( E(n, \delta) \geq  \varepsilon) + \sum_{k=1}^{2n}  \frac{ \varepsilon ^{k}}{k} \leq 3 \varepsilon.  \end{eqnarray*} 
   The second point is similar.  Conditionally on $ \mathcal{U}_{2n}$, if $w$ is a vertex of degree larger than $\delta n$ then the probability that any fixed vertex $v$ of degree $d$ is not linked to $w$ is bounded above by $ ( 1- \frac{ \delta }{2})^{d}$. Since there are deterministically less than $ \frac{2}{\delta}$ vertices in $ \mathrm{V}_{  \delta n}^{(n)}$, the complement probability of \eqref{eq:pairgrand} is bounded above by
   \begin{eqnarray*}
    \frac{2}{\delta} \mathbb{E}\left[ \sum_{ d \in \mathcal{U}_{2n} } \left(1 - \frac{\delta}{2}\right)^{d}  \ind_{ d \geq A} \right] = \frac{2 }{\delta} \sum_{k=A}^{2n} \frac{\left(1 - \frac{\delta}{2}\right)^{k}}{k} \leq \left(\frac{2}{\delta}  \right)^{2} \left(1- \frac{\delta}{2}\right)^{A} \frac{1}{A}. \end{eqnarray*}
For fixed $\delta>0$, the right-hand side can be made arbitrary small by taking $A$ large and this proves the second point of the lemma.
\end{proof}

\begin{proof}[Proof of Lemma \ref{lem:decoupling}]
This is essentially a corollary of the very general results presented in \cite{arratia2000limits}. More precisely, the convergence in law of the small cycle lengths count is a corollary of their Theorem 3.1, whereas the convergence of the large cycle lengths is the convergence towards $ \mathsf{PD}(1)$ recalled in the introduction. To see that they are asymptotically independent, just notice that for fixed $k_{0}$, conditionally on $C_{1}^{(n)}=c_{1}, \dots , C_{k_{0}}^{(n)}=c_{k_{0}}$ for fixed values $c_{1}, \dots c_{k_{0}} \geq 0$, the law of $C_{k_{0}+1}^{(n)}, \dots , C_{2n}^{(n)}$ is distributed as $P_{k_{0}+1}, \dots, P_{2n}$ conditioned on $ \sum_{i=k_{0}+1}^{2n} = 2n - \sum_{i=1}^{k_{0}} i c_{i}$ which again falls in the general ``conditioning relation and logarithmic condition'' of \cite{arratia2000limits}, for which we can thus apply their Theorem 3.2 giving in particular the convergence towards Poisson--Dirichlet.
\end{proof}
 
\section{Dynamical exploration of random surfaces} \label{sec:peeling}
In this section we describe the ``dynamical'' exploration of the random surface $  \mathbf{M}_{ \mathcal{P}}$. Our approach is close to the peeling process used in the theory of random planar maps, see \cite{Bud15,curienpeeling}. 

\subsection{Exploration of surfaces and topology changes} \label{sec:topologychanges}
We fix a configuration $ \mathcal{P}$ of (oriented) polygons with $| \mathcal{P}|=n$ whose sides have been labeled from $1$ to $2n$, and  $\omega \in \mathcal{I}_{2n}$ a pairing of its edges. Since everything we describe below is deterministic, we do not need to assume here that $\omega$ is random.

We will construct step by step the discrete surface $ \mathbf{M}_{ \mathcal{P}}$ that is created by matching the edges 2 by 2 according to $\omega$. More precisely, we will create a sequence $$ S_{0} \to S_{1} \to  \dots \to S_{n} = \mathbf{M}_{ \mathcal{P}}$$ of ``combinatorial surfaces''  where $S_{0}$ is made of the set of labeled polygons whose perimeters are specified by $ \mathcal{P}$ and where we move on from $ S_{i} $ to $S_{i+1}$ by identifying two edges of the pairing $\omega$. More specifically, $S_{i}$ will be a union of labeled maps with distinguished faces called the holes (they are in light green in the figures below). The holes are made of the edges which are not yet paired. The set of these edges will be called the boundary of the surface and be denoted by $ \partial S_{i}$. Clearly, if $2| \mathcal{P}|= |\partial S_{0}|=2n$, we thus have 
$$ | \partial S_{i}| =  2n-2i.$$
\begin{figure}[!h]
 \begin{center}
\begin{overpic}[width=15cm]{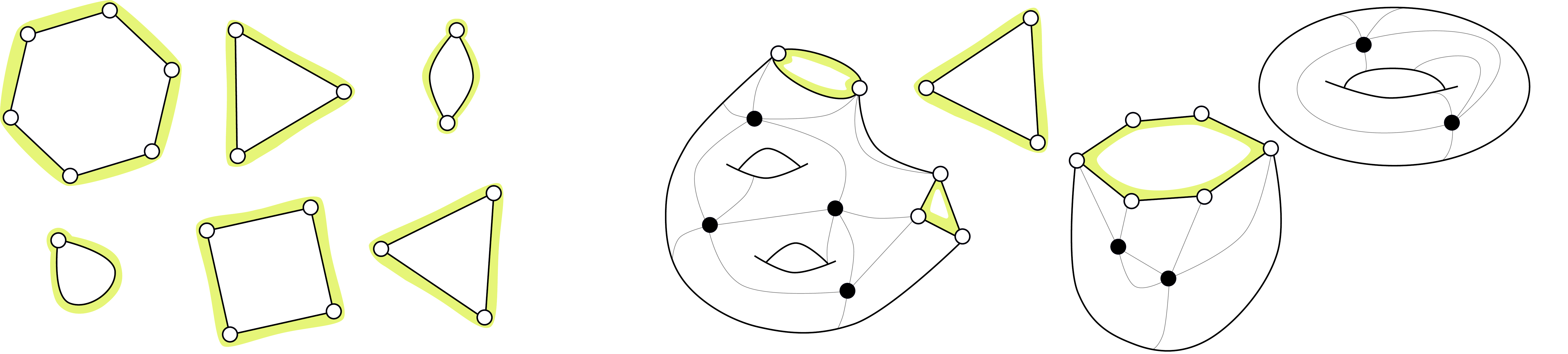}
\put (34,12) {etc.}
\put (98,12) {etc.}
\end{overpic}
 \caption{Starting configuration (on the left) and a typical state of the exploration (on the right). Here and later the labeling of the oriented edges does not appear for the sake of visibility. The final vertices of the graph are black dots whereas ``temporary'' vertices are in white. Notice on the right-hand side that $S_{i}$ contains a closed surface without boundary: if this happens, the final surface $S_{n} = \mathbf{M}_{ \mathcal{P}}$ is disconnected.}
 \end{center}
 \end{figure}
 
To go from $S_{i}$ to $S_{i+1}$ we select an edge on $ \partial S_{i}$ which we call \emph{the edge to peel} (in red in the figures below) and identify it with its partner edge in $\omega$ (in green in the figures below), also belonging to $\partial S_{i}$. We now describe  the possible outcomes of  the peeling of one edge. The reader should keep in mind that our surfaces are always labeled and oriented and that when identifying two edges we glue them in a way compatible with the orientation. We shall also pay particular attention to the process of creation of vertices. Indeed, the initial vertices of the polygons are not all vertices in the final map $ \mathbf{M}_{\mathcal{P}}$: in the figures below those ``temporary'' vertices belonging to holes are denoted by white dots, whereas actual vertices of $ \mathbf{M}_{\mathcal{P}}=S_{n}$ are denoted by black dots and called ``true'' vertices. 

 \begin{figure}[H]
  \begin{center}
  \begin{overpic}[width=13cm]{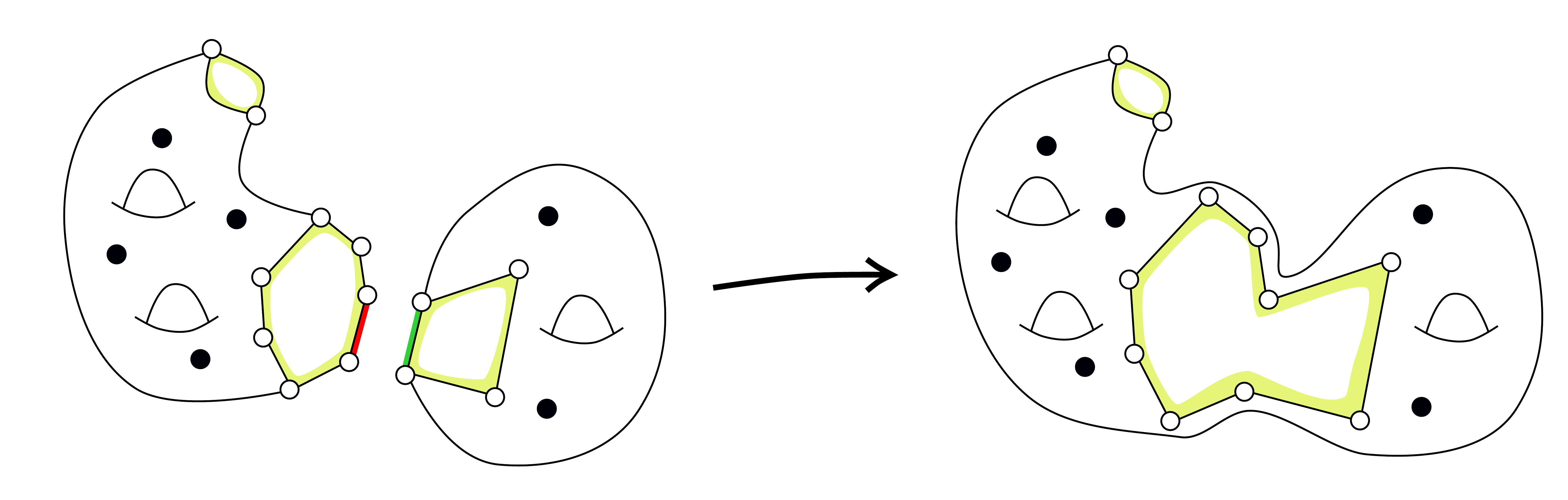}
   \put (19,12) {$p_1$}
   \put (28,10) {$p_2$}
   \put (73.5,9.5) {$p_1+p_2-2$}
  \end{overpic}
  \caption{ \label{fig:case1}If we identify two edges of different components, the holes and the components merge (their genuses and number of true vertices add up). In particular the perimeter of the resulting boundary is the sum of the perimeters of the former boundaries minus $2$.}
  \end{center}
  \end{figure}
  
   \begin{figure}[H]
  \begin{center}
  \begin{overpic}[width=12cm]{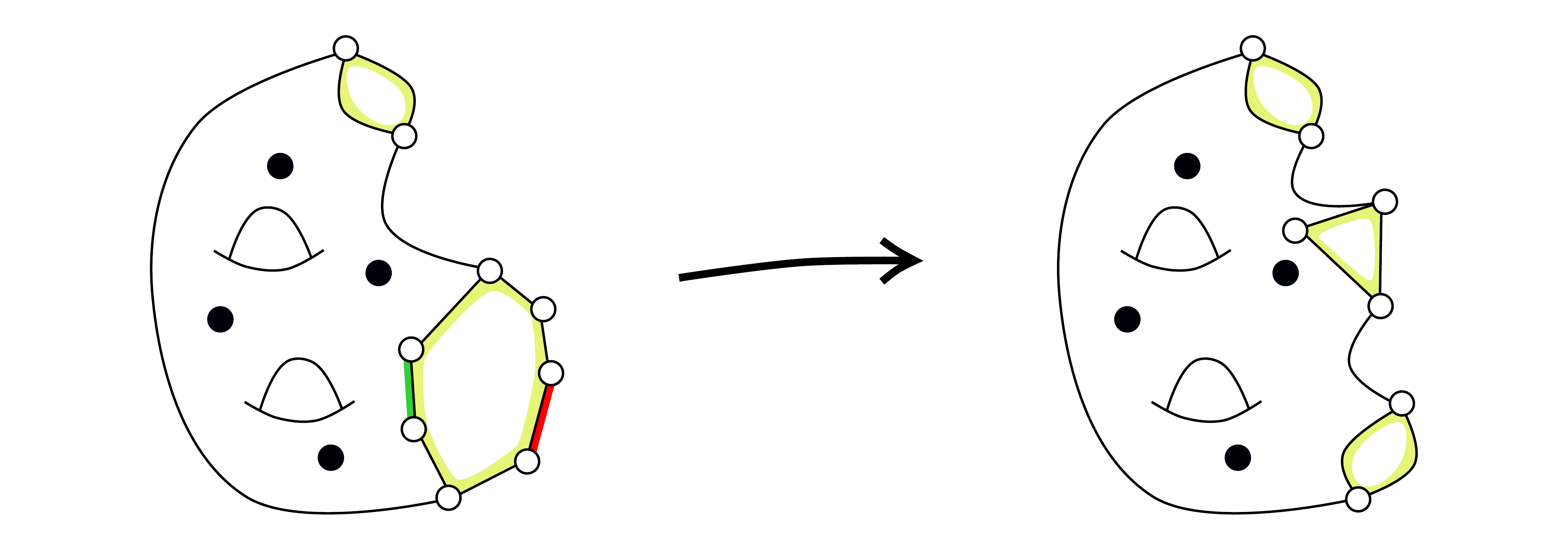}
   \put (29,10) {$p_1$}
   \put (89,19) {$p_1'$}
   \put (90.5,4) {$p_1''$}
  \end{overpic}  
  \caption{\label{fig:case2}If we identify two edges on the same hole, then this hole splits into two holes of perimeters $p'_{1}$ and $p''_{1}$ such that $p'_{1} + p''_{1}+2$ is the perimeter of the initial boundary.}
  \end{center}
  \end{figure}
  
     \begin{figure}[H]
  \begin{center}
  \begin{overpic}[width=11cm]{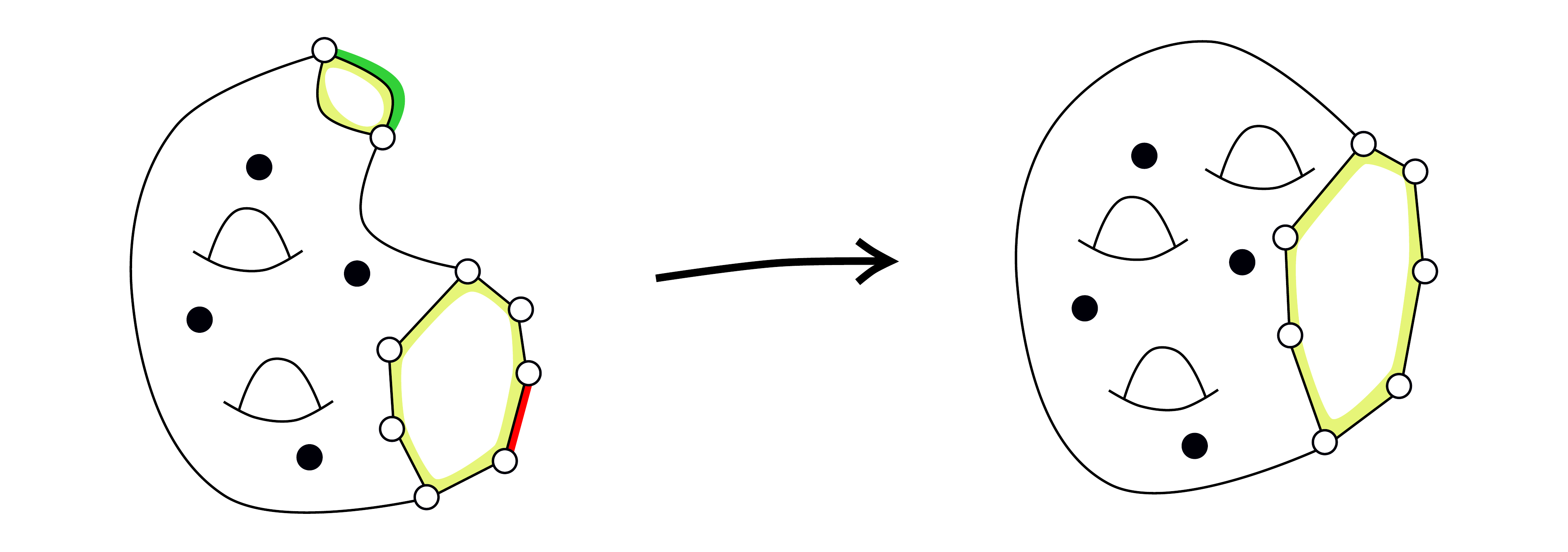}
   \put (28,10) {$p_1$}
   \put (26.5,30) {$p_2$}    
   \put (91,14) {$p_1+p_2-2$}    
  \end{overpic}    
  \caption{ \label{fig:case3} If we identify two edges of different holes belonging to the same component, then the holes merge, adding one unit to the genus of the map. As in Fig.~\ref{fig:case1}, the perimeter of the new hole is the sum of the perimeters of the initial holes minus $2$.}
  \end{center}
  \end{figure}

We now describe the possible steps that yield to creation of true vertices: 

     \begin{figure}[H]
  \begin{center}
  \begin{overpic}[width=15cm]{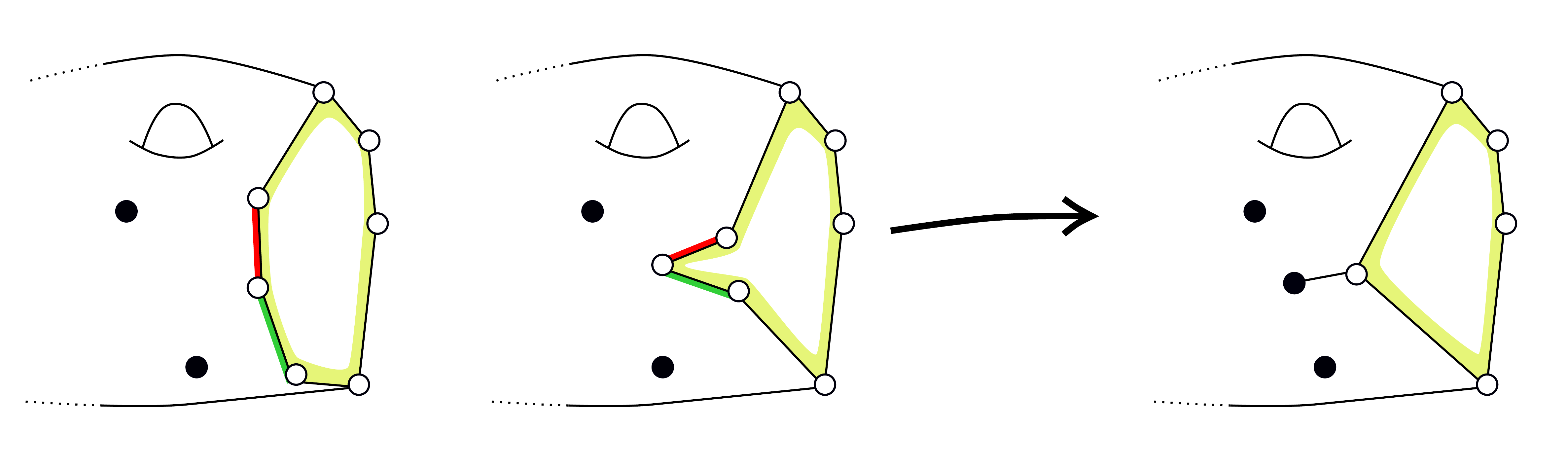}
  \put (20,13) {$p$}
  \put (27,13) {{\huge $=$}}
  \put (49.5,13) {$p$}
  \put (89.5,12) {$p-2$}
  \end{overpic}
  \caption{\label{fig:case4} As a special case of Fig.~\ref{fig:case2}, if we identify two neighboring edges on the same hole, then one of the two holes created has perimeter $0$: we have formed a true vertex.}
  \end{center}
  \end{figure}
  
       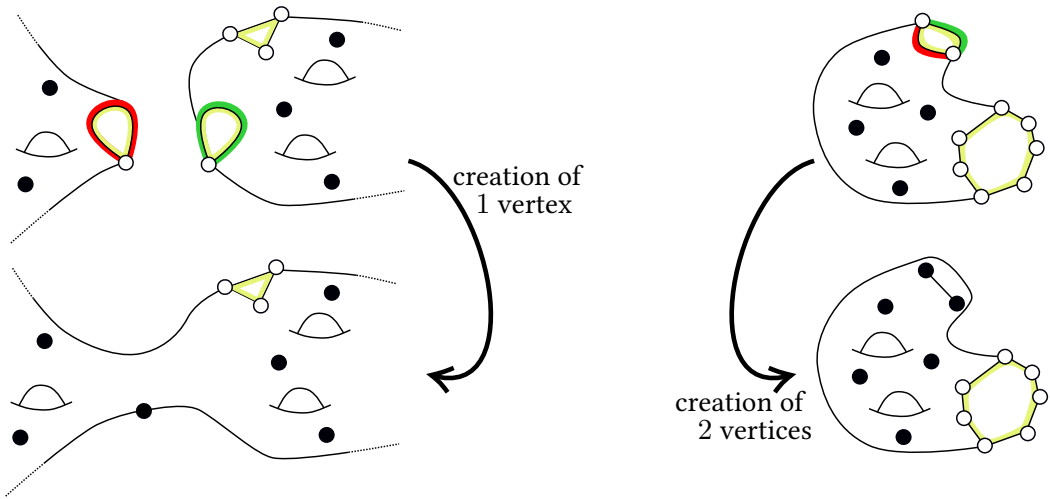
\begin{figure}[H]
  \begin{center}
  \begin{overpic}[width=15cm]{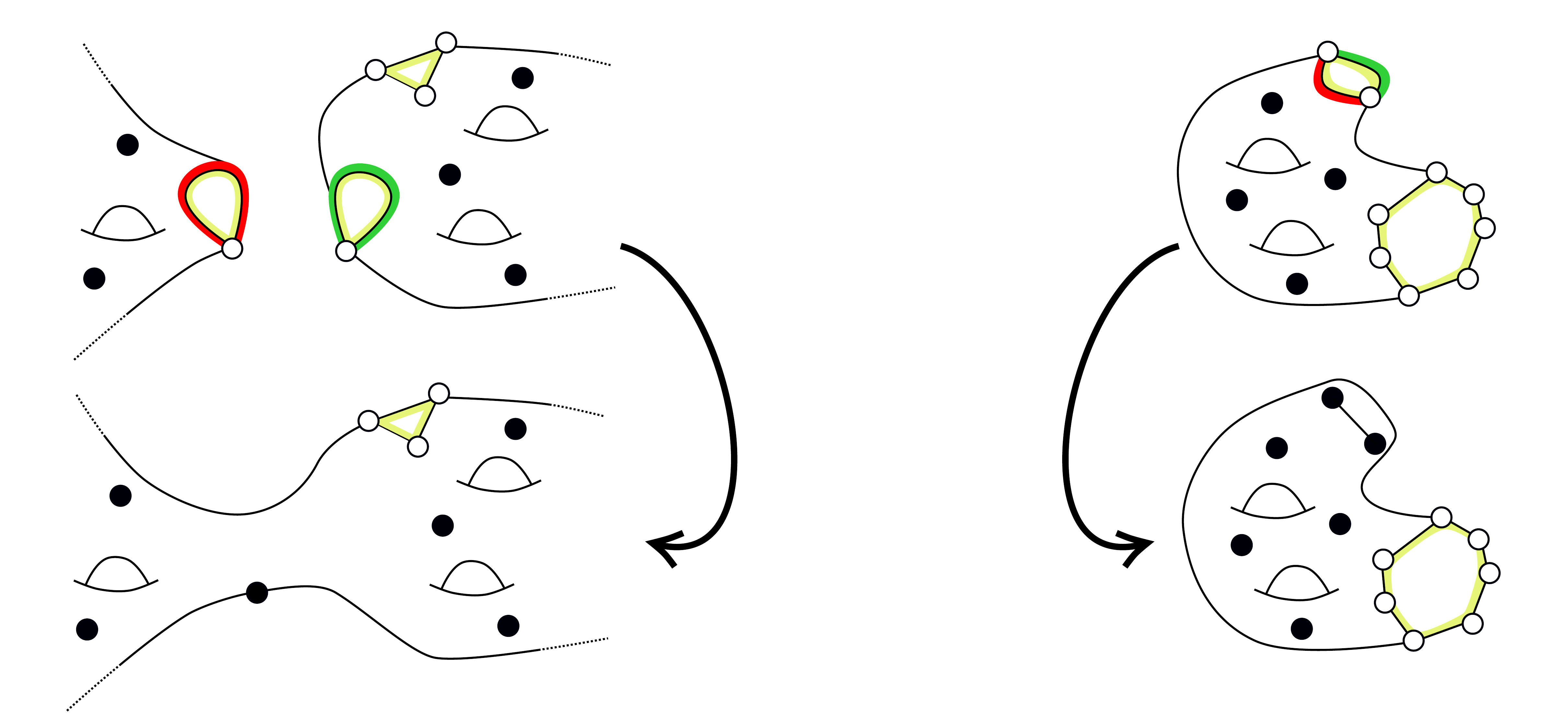}
   \put (43.5,28.5) {creation of}
   \put (45.5,26) {$1$ vertex}   
  \put (63,8.5) {creation of}
   \put (65,6) {$2$ vertices}   
  \end{overpic}
  \caption{\label{fig:case5}The two other ways to create true vertices: if we identify (left) two loops (holes of perimeter 1), then we create one true vertex. If we identify (right) the two edges of a hole of perimeter $2$, then we create two true vertices.}
  \end{center}
  \end{figure}

\subsection{Peeling exploration}
We now move on to our random setting and suppose that $\omega \in \mathcal{I}_{2n}$ is independent of the starting configuration of labeled polygons. On top of $\omega$, the sequence $S_{0} \to S_{1} \to \dots \to  S_{n}$ depends on an \emph{algorithm} called the peeling algorithm which is simply a way to pick the next edge to peel $ \mathcal{A}(S_{i}) \in \partial S_{i}$. 
The function $ \mathcal{A}(S_{i})$ can be deterministic or may use another  source of randomness, as long as it is independent of $\omega$: we call such an algorithm \emph{Markovian}. Highlighting the dependence in $ \mathcal{A}$, we can thus form the random exploration sequence 
$S_{0}^{ \mathcal{A}} \to S_{1}^{ \mathcal{A}} \to \dots \to S_{n}^{ \mathcal{A}} =  \mathbf{M}_{ \mathcal{P}}$ by starting with $S_{0}^{ \mathcal{A}}$, the initial configuration made of the labeled polygons whose perimeters are prescribed by $ \mathcal{P}$. To go from $ S_{i}^{ \mathcal{A}}$ to $ {S}_{i+1} ^{ \mathcal{A}}$, we perform the identification of the edge $ \mathcal{A}(S_{i})$ together with its partner in the pairing $ \omega$. When $\omega$ is uniform and $ \mathcal{A}$ is Markovian, the sequence $(S_{i}^{ \mathcal{A}} : 0 \leq i \leq n)$ is a simple (inhomogeneous) Markov chain:
\begin{proposition} \label{prop:markov} If the gluing $\omega$ is uniformly distributed and independent of the labeled polygons of $ \mathcal{P}$, then for any Markovian algorithm $ \mathcal{A}$, the exploration $(S_{i}^{ \mathcal{A}})_{ 0 \leq i \leq n}$ is an inhomogeneous Markov chain whose probability transitions are described as follows. Conditionally on $ S_{i}^{ \mathcal{A}}$ and on $ \mathcal{A}(S_{i}^{ \mathcal{A}})$, we pick $E_{i}$ uniformly at random among the $2(n-i)-1$ edges of $\partial S_{i}^{ \mathcal{A}} \backslash \{\mathcal{A}(S_{i}^{ \mathcal{A}})\}$, and identify $\mathcal{A}(S_{i}^{ \mathcal{A}})$ with $E_{i}$.
\end{proposition}
\proof It suffices to notice by induction that at each step $i \geq 0$ of the exploration, conditionally on $ S_{i}$, the pairing $\tilde{\omega}$ of the (unexplored) edges of $\partial S_{i}^{ \mathcal{A}}$ is uniform. Hence, if an edge $\mathcal{A}(S_{i}^{ \mathcal{A}})$ is picked independently of $\tilde{\omega}$ then its partner is a uniform edge on $ \partial S_{i}^{ \mathcal{A}} \backslash \{ \mathcal{A}(S_{i}^{ \mathcal{A}})\}$.  \endproof 

\begin{remark}[Uniform peeling and split/merge dynamics] Probably the most obvious Markovian peeling algorithm is the following. For $i \geq 0$, given the discrete surface $S_{i}$, we pick the next edge to peel uniformly at random on  $ \partial S_{i}$ (independently of the past operations and of the gluing of the edges).\\
Although very natural, we shall not use this peeling algorithm in this paper. We note however that the Markov chain this algorithm induces on the set of perimeters of the holes of $ S_{i}$ is very appealing. If $\{p_{1}, \dots , p_{k}\}$ is the configuration of the perimeters at time $i$, the next state is obtained by first sampling independently two indices $I,J \in \{1,2, \dots, k\}$ proportionally to $p_{1}, \dots, p_{k}$. If $I \ne J$, then we replace the two numbers $p_{I}$ and $p_{J}$ by $p_{I}+p_{J}-2$. If $I=J$, we replace $p_{I}$  by a uniform splitting of $p_{I}-2$ into $\{\{0,p_{I}-2\}, \{1, p_{I}-2\}, \dots, \{p_{I}-2,0\}\}$. 
This is a discrete version of the split-merge dynamic, which preserves the Poisson--Dirichlet law, considered in \cite{diaconis2004,schramm2005compositions}, except that we have a deterministic ``erosion'' of $-2$ at each step in the above dynamic. 
\end{remark}

The strength of the above proposition is that, as for planar maps \cite{curienpeeling}, we can use \emph{different} algorithms $ \mathcal{A}$ to explore the \emph{same} random surface $ \mathbf{M}_{ \mathcal{P}}$ and then to get \emph{different} types of information. We will see this motto in practice in the following sections. When exploring our random surface with a given peeling algorithm, we will always write $( \mathcal{F}_{i})_{0 \leq i \leq n}$ for the canonical filtration generated by the exploration.
\subsection{Controlling loops and bigons for good configurations}

In this subsection, we give rough bounds on the number of holes of perimeter $1$ and $2$ that appear during Markovian explorations of good sequences of configurations. We will call \emph{loop} (resp. \emph{bigon}) a hole of perimeter $1$ (resp. $2$). The purpose of such estimates will be to bound the number of steps at which the cases of Figure \ref{fig:case5} occur during the exploration, which will be useful several times in the next sections.

Recall from the introduction the definition of a good sequence $ (\mathcal{P}_{n})_{n \geq 1}$ of configurations and assume we explore $ \mathbf{M}_{ \mathcal{P}_{n}}$ using a Markovian algorithm. With an implicit dependence on $ \mathcal{P}_{n}$ and $n$, denote by $ (L_{i} : 0 \leq i \leq n)$  and $(B_{i} : 0 \leq i \leq n)$ the number of loops and bigons in  $\partial S_{i}$ during the exploration. We also write $ \pi(i)$ for the perimeter of the hole to which the peeled edge at time $i$ belongs. 
\begin{proposition}[Bounds on loops and bigons] \label{prop:loopbigon} For any configuration and any Markovian peeling algorithm we have the stochastic dominations
$$ \sup_{0 \leq i \leq n} L_{i}  \leq L_{0} + X_{n}, \qquad \sup_{0 \leq i \leq n} B_{i} \leq B_{0} + L_{0} + X_{n} \quad \mbox{ and } \qquad \sum_{i=0}^{n-1} \ind_{ \pi(i)=2} \leq B_{0}+L_{0} +X_{n}$$
 $$ \mbox{ where} \quad  X_{n} = 2 \sum_{i=1}^{n} \mathrm{Bern} \left( \frac{4}{i} \right),$$ with independent Bernoulli random variables. In particular, if $( \mathcal{P}_{n})_{n \geq 1}$ is a good sequence of configurations then we have $\sup L_{i} = o ( \sqrt{n})$, also $\sup_{i} B_{i} =o(n)$ and $\sum_{i=0}^{n-1} \ind_{ \pi(i)=2} = o(n)$ with very high probability. Finally, for any good sequence of configurations, for any $ \varepsilon>0$, we have
 $$  
 \frac{1}{ \sqrt{n}}  \mathbb{E}\left[\sum_{i=0}^{[(1- \varepsilon)n]} \ind_{\pi(i)=1}\right] \xrightarrow[n\to\infty]{} 0.$$
 \end{proposition}
\proof
Let us describe first the variation of the number of bigons and loops during a peeling step using Section \ref{sec:topologychanges}. One can create loops or bigons if at step $i \geq 0$ we are peeling on a $p$-gon with $p \geq 3$ and identify the peeled edge with the second edge on its right or left along the same hole (creation of one loop) or the third edge on its right or left along the same hole (creation of a bigon). When $p=4$, peeling the second edge on the right or left are the same event but it yields to creation of $2$ loops, and similarly when $p=6$ for bigons. In any of these cases we have $\Delta L_{i} \leq 2$ and $\Delta B_{i} \leq 2$, and the conditional (on $\mathcal{F}_i$) probability of those events is bounded above by 
$$ \frac{4}{2(n-i)-1}.$$
In the case $p=2$, recall that the identification of both sides of a bigon only results in the disappearance of the bigon and the creation of $2$ true vertices.
Otherwise, the $i$-th peeling step identifies the peeled hole of perimeter $\pi(i)$ with another hole of perimeter we denote by $\xi(i)$. This results in the creation of a hole of perimeter $\pi(i)+\xi(i)-2$ (in the case $\pi(i)= \xi(i)=1$, i.e.~when we identify two loops, we just create a true vertex). 

\textsc{First domination.} By the above description, appart from the first kind of events where we identify two edges on the same hole, we notice that we always have $\Delta L_{i} \leq 0$. Hence, the distribution of $\Delta L_{i}$ conditionally on $ \mathcal{F}_{i}$ is stochastically dominated by 
$$ 2 \cdot \mathrm{Bern} \left( \frac{2}{2(n-i)-1}\right),$$
where  $\mathrm{Bern}(p)$ is a Bernoulli variable of parameter $p$. Clearly $\sup_{i <n}(L_{i}-L_{0})$ is stochastically dominated by $Z = 2\sum_{i =0}^{n-1}\mathrm{Bern} \left( \frac{2}{2(n-i)-1}\right)$ with independent Bernoulli variables, and this proves the first claim.

\textsc{Second domination.} We now focus on $\Delta (L_{i}+B_{i})= L_{i+1}-L_{i} + B_{i+1}-B_{i}$. Again, appart from the cases where we identify two edges on the same hole which result in $ \Delta (L_{i}+B_{i}) \leq 2$, we always have $\Delta (L_{i}+B_{i}) \leq 0$. Hence, conditionally on $ \mathcal{F}_{i}$, the distribution of $\Delta (L_{i}+B_{i})$ is stochastically dominated by 
$$ 2 \cdot \mathrm{Bern} \left( \frac{4}{2(n-i)-1}\right),$$ and the same reasoning as above yields the second point of the Proposition.

\textsc{Third domination.} With a closer look at the cases above, we see that the conditional distribution of $\ind_{\pi(i)=2} + \Delta (L_{i}+B_{i})$ given $ \mathcal{F}_{i}$ is again stochastically dominated by $ 2 \cdot \mathrm{Bern}(  \frac{4}{2(n-i)-1})$. Summing over $ 0 \leq i \leq n-1$ yields the result.

\textsc{Last convergence.} For the last convergence of the proposition, remark that
\begin{equation}\label{eqn:peeled_loops_one_step}
\mathbb{E}[ \ind_{ \pi(i)=1}+ \Delta L_{i} \mid \mathcal{F}_{i}] \leq \frac{2 }{2(n-i)-1} +  \frac{\ind_{\pi(i)=1} \cdot 2 B_{i}}{2(n-i)-1},
\end{equation}
 where the first term in the right hand side comes from splitting a hole of perimeter $p \geq 3$, and the second term comes from the case where a loop is peeled and identified with a bigon to produce a new loop. Taking expectation, after summing over $0 \leq i \leq (1- \varepsilon)n$, we get for $n$ large
 \begin{eqnarray*}\mathbb{E}\left[\sum_{i=0}^{[(1- \varepsilon)n]} \ind_{ \pi(i)=1}\right] &\leq& L_{0}+ \sum_{i=0}^{[(1- \varepsilon)n]} \frac{2}{2(n-i)-1} +  \mathbb{E}\left[\sum_{i=0}^{[(1- \varepsilon)n]} \frac{2}{2(n-i)-1} \ind_{ \pi(i)=1} \sup_{0 \leq j \leq n} B_{j} \right]  \\ &\leq& L_{0}+ O(\log n) + 2 n(1- \varepsilon)\mathbb{P} \left( \sup_{0 \leq j \leq n} B_{j} \geq  \varepsilon^{2} n \right) + \varepsilon \mathbb{E}\left[\sum_{i=0}^{[(1- \varepsilon)n]} \ind_{ \pi(i)=1}\right].
  \end{eqnarray*}
By our second domination $\sup_{i} B_{i} \leq B_{0}+L_{0}+X_{n}$ and our goodness assumption $L_{0}+B_{0} = o (\sqrt{n}) + o(n) = o(n)$, we easily deduce that $\mathbb{P}( \sup_{0 \leq j \leq n} B_{j} \geq  \varepsilon^{2} n)$ decreases faster than $n^{-2}$. Therefore, after re-arranging the last display, we obtain $(1- \varepsilon)\mathbb{E}\left[\sum_{i=0}^{[(1- \varepsilon)n]} \ind_{ \pi(i)=1}\right] = o( \sqrt{n})$ as required.
  \endproof

\section{Peeling the minimal hole and number of vertices}\label{sec:minimalhole}
In this section we study the \emph{number of vertices} of $ \mathbf{M}_{ \mathcal{P}}$ and prove Theorem \ref{thm:vertices}. The corner stone of the proof is to explore our surface using the following algorithm.\\

\fbox{ \begin{minipage}{14cm}
\textbf{Algorithm $ \mathcal{H}$ or peeling the minimal hole}: Given $S_{i}$ for $0 \leq i <n$, the next edge to peel $ \mathcal{H}( S_{i})$ is one of the edges which belong to a \emph{hole of minimal perimeter}. If there are multiple choices, we pick the edge having the minimal label.
\end{minipage}}\medskip

It is clear that the above algorithm is Markovian and so the transitions of the Markov chain $(S_{i})_{i \geq 0}$ are described by Proposition \ref{prop:markov}. In the rest of this section, we always explore our random surfaces using  the above algorithm.
\subsection{Towards a single hole}
 \label{sec:Hprocess}
Fix a configuration $ \mathcal{P}$ with $ |\mathcal{P}|=n$ and let us explore the random surface $ \mathbf{M}_{ \mathcal{P}}$ using the peeling algorithm $ \mathcal{H}$. All our notations below depend implicitly on $ \mathcal{P}$.  A first observation concerns the number of holes during this exploration. If $H_{i} $ denotes the number of holes of $S_{i}$ then $H_{0} = \# \mathcal{P}$ and it is easy to see from the possible topology changes (Section \ref{sec:topologychanges}) that $$H_{i+1}-H_{i} \in \{-2,-1,0,+1\}.$$
Furthermore, if $H_{i+1}-H_{i} \in \{ 0,1\}$, then we have identified two edges of the same hole during the $i$-th peeling step. Since the minimal perimeter of a hole is at most $ \lfloor | \partial S_{i} | / {H}_{i} \rfloor $, if $ \mathcal{F}_{i}$ is the $\sigma$-field of the past exploration up to time $i$, we get from Proposition \ref{prop:markov} that 
 \begin{eqnarray} \label{eq:driftH} \mathbb{P}( H_{i+1}-H_{i} \in \{ 0,1\} \mid  \mathcal{F}_{i}) \leq \frac{ \left\lfloor |\partial S_{i}|/H_{i}\right\rfloor-1}{ |\partial S_{i}|-1 } \leq \frac{1}{H_{i}}.  \end{eqnarray}
In other words, as long as $H_{i}$ is large, the process $H$ undergoes a strong negative drift (it decreases by $1$ or $2$ with a high probability at each step).  If $\tau = \inf\{ i \geq 0 : H_{i} =1\}$ is the first time at which $S_{i}$ has a \emph{single hole}, then the above inequality together with easy probabilistic estimates show that $H$ hits the value $1$ "almost as soon as possible".
\begin{lemma}[$\tau \approx \# \mathcal{P}$] \label{lem:tau} For every $ \varepsilon>0$ we can find $C_{ \varepsilon}>0$ such that uniformly in $ \mathcal{P}$, we have 
$$ \mathbb{P}(\tau \geq (1+ \varepsilon) \# \mathcal{P} + C_{ \varepsilon}) \leq 	\varepsilon.$$
\end{lemma}
\proof Fix $ \varepsilon>0$ and pick $A \geq 1$ large so that $ \frac{A}{A-2} \leq (1 + \varepsilon)$. Let us define $\tau_{A} = \inf\{i \geq 0 : H_{i} \leq A\}$. Clearly, by the Markov property of the exploration and the above calculation, as long as $i \leq \tau_{A}$, the increments of $H$ are stochastically dominated by independent random variables $\xi_{i}$ with law $ \mathbb{P}(\xi=1) = \frac{1}{A}$ and $ \mathbb{P}(\xi = -1) = 1- \frac{1}{A}$. As a result, $\tau_{A}$ is stochastically dominated by $\alpha( \# \mathcal{P})$ the hitting time of $A$ by a random walk with i.i.d.~increments of law $\xi$ started from $\# \mathcal{P}$. Now, as soon as  $H_{i}$ drops below $A$ and until time $\tau$, we can stochastically bound its increments by those of a simple symmetric random walk on $ \mathbb{Z}$. If $\beta$ denotes the hitting time of $1$ by a simple symmetric random walk on $ \mathbb{Z}$ started from $A$, we can thus write the stochastic inequality $$ \tau \leq \alpha( \#\mathcal{P})+ \beta,$$ where the last two variables are independent (but we shall not use it). Notice that by the law of large numbers $\alpha ( \# \mathcal{P}) \sim \# \mathcal{P}  \frac{A}{A-2}$ as $\# \mathcal{P} \to \infty$. By our choice of $A$, this implies that $ \mathbb{P}( \alpha(\# \mathcal{P})\geq (1+ \varepsilon) \# \mathcal{P}) \leq \varepsilon/2$ for all $ \#\mathcal{P}$ larger than some $p_{0}\geq 1$. We then fix $ C_{ \varepsilon}$ large so that $ \mathbb{P}( \beta  \geq  C_{ \varepsilon}/2) \leq  \varepsilon/4$ and $\mathbb{P}(   \sup_{p \leq p_{0}} \alpha(p)\geq  C_{ \varepsilon}/2) \leq \varepsilon/4$. When doing so,  the statement of the lemma holds true. \endproof 

\paragraph{Back to one hole and unicellular maps.}  Performing the exploration of the surface with algorithm $ \mathcal{H}$ until time $\tau$ is particularly convenient. Indeed, $S_{\tau}$ has a single hole of perimeter $2(n-\tau)$ and by the Markov property of the exploration, to get the final surface $ \mathbf{M}_{ \mathcal{P}}$, one just needs to glue the edges of this hole using an independent uniform pairing of its edges. This is obviously a particular case of our construction of random surface for the case where $ \mathcal{P} = \{2(n-\tau)\}$ is made of a single polygon. The (obviously connected!) surface 
$ \mathbf{M}_{\{2k\}}$ is known as a \emph{unicellular map} with $k$ edges since it is a map having a single face. Using this observation, we can draw two conclusions: first, the surface $ \mathbf{M}_{ \mathcal{P}}$ is connected if and only if $S_{\tau}$ is connected. Second, if we  focus on the number of vertices $ \mathsf{V}_{ \mathcal{P}}$ of the underlying surface we can write 
 \begin{eqnarray} \label{eq:Xtau} \mathsf{V}_{ \mathcal{P}} = X_{\tau} + \mathsf{V}_{\{2(n-\tau)\}},  \end{eqnarray}
where $X_{\tau}$ is the number of ``true'' vertices of $S_{\tau}$ (i.e.~actual vertices of $ \mathbf{M}_{ \mathcal{P}}$) and $ \mathsf{V}_{\{ 2(n-\tau)\}}$ is the number of vertices of $ \mathbf{M}_{\{2(n-\tau)\}}$ a unicellular map with $n-\tau$ edges, which is conditionally on $\tau$, independent of $S_{\tau}$. Let us use these two remarks to study the connectedness and the number of vertices of $ \mathbf{M}_{ \mathcal{P}}$.
\subsection{Connectedness} 
Let us focus first on the connectivity of our surface.  We fix a good sequence $ (\mathcal{P}_{n})_{n \geq 1}$ of polygonal configurations and for each $n$, perform the exploration of $ \mathbf{M}_{ \mathcal{P}_{n}}$ using algorithm $ \mathcal{H}$. We write $\tau \equiv \tau_{n}$ to highlight the dependence in $ \mathcal{P}_{n}$ and $n$. Remark that if $( \mathcal{P}_{n})$ is a good sequence of configurations, then we have few loops and bigons and so    \begin{eqnarray}\frac{ \# \mathcal{P}_{n}}{n} \leq \frac{2}{3} \label{eq:2/3} \end{eqnarray} asymptotically as $ n \to \infty$. In particular, by Lemma \ref{lem:tau}, when performing the peeling using algorithm $ \mathcal{H}$, we reach a configuration having a single hole typically before $2n/3$ out of the $n$ peelings steps. By the last discussion, $ \mathbf{M}_{ \mathcal{P}_{n}}$ is connected if and only if we managed to reach time $\tau_{n}$ without having disconnected a piece of our surface en route. We prove that for a good sequence of configurations, this situation is very likely:
\begin{proposition} \label{prop:connected} If $( \mathcal{P}_{n})_{n \geq 1}$ is a good sequence of configurations, then $ \mathbf{M}_{ \mathcal{P}_{n}}$ is connected with high probability as $ n \to \infty$. \end{proposition}
\begin{remark} It is an exercise to prove that if $ \mathcal{P}$ either contains more that $ \varepsilon \sqrt{| \mathcal{P}|}$ loops or more than $ \varepsilon | \mathcal{P}|$ bigons, then there exists a constant $c_{ \varepsilon}>0$ such that $ \mathbb{P}( \mathbf{M}_{ \mathcal{P}} \mbox{ is not connected}) \geq c_{ \varepsilon}$ uniformly for all configurations.  Our assumptions for connectedness are thus optimal.
\end{remark}
\proof
We recall that $S_{0} \to S_{1} \to \cdots \to S_{n} = \mathbf{M}_{ \mathcal{P}_{n}}$ is the exploration of  $\mathbf{M}_{ \mathcal{P}_{n}}$ using the peeling algorithm $ \mathcal{H}$. By the list of the possible outcomes of a peeling step, we see that if $S_{\tau_{n}}$ is not connected, then at some time $0 \leq i < \tau_{n}$, we have performed a peeling step identifying either the two sides of a bigon or two loops together. 
We shall see that such operations are unlikely to happen before $\tau_{n}$ if we start with few loops and bigons. Recall that $ L_{i}$ (resp.~$B_{i}$) is the number of loops (resp.~bigons) in  $\partial S_{i}$ and that $\pi(i)$ is the perimeter of the hole peeled at time $i$. 
For each $0\leq i<n $, conditionally on the past of the exploration, the probability of either closing the two sides of a bigon or identify two loops together is 
 \begin{eqnarray*} \frac{1}{2(n-i)-1} \left(\ind_{\pi(i)=2} + L_{i}\ind_{\pi(i)=1}\right).  \end{eqnarray*}
Summing for all $0 \leq i \leq 3n/4$ and taking expectations, we deduce that there is a constant $C>0$ such that, for every $ \varepsilon >0$, we have 
\begin{eqnarray*} &&\mathbb{E}\left[ \sum_{i=0}^{[3n/4]} \frac{1}{2(n-i)-1} \left(\ind_{\pi(i)=2} + L_{i}\ind_{\pi(i)=1}\right) \right]\\  &\leq&  \frac{C}{n} \mathbb{E}\left[\sum_{i=0}^{[3n/4]} \ind_{\pi(i)=2}\right] + \frac{C}{n}\mathbb{E}\left[\sum_{i=0}^{[3n/4]} \ind_{\pi(i)=1} \sup_{j\leq n} L_{j}\right] \\
& \underset{ \mathrm{Prop.} \ref{prop:loopbigon}}{\leq}{} & o(1) + \frac{C}{n}\left( \varepsilon \sqrt{n}\mathbb{E}\left[\sum_{i=0}^{[3n/4]} \ind_{\pi(i)=1}\right] + n^{2} \mathbb{P}( \sup_{j \leq n} L_{j} \geq \varepsilon \sqrt{n}) \right) \\
& \underset{ \mathrm{Prop.} \ref{prop:loopbigon}}{\leq}{} &  o(1)+ C n \times \mathbb{P}( \sup_{j \leq n} L_{j} \geq \varepsilon \sqrt{n}) .  \end{eqnarray*}
But by Proposition \ref{prop:loopbigon} again, $C n \times \mathbb{P}( \sup_{j \leq n} L_{j} \geq \varepsilon \sqrt{n})$ goes to $0$ as $n \to \infty$ since by our goodness assumption we have $L_{0}= o( \sqrt{n})$. Hence the probability to perform an event which may yield to disconnection of the surface before time $3n/4$ is going to $0$. Since $\tau_{n} \leq 3n/4$ with high probability, by Lemma \ref{lem:tau}, the result is proved.\endproof 

It follows from the above proof that it is very unlikely that two loops are glued together before time $\tau_{n}$. If so, the number of holes cannot decrease by more than $1$ at each step and so with high probability we have 
  \begin{eqnarray} \label{eq:taunborneinf} \tau_{n} \geq \# \mathcal{P}_{n}-1,  \end{eqnarray}
  which complements the upper bound of Lemma \ref{lem:tau}. The above proof of connectedness of $ \mathbf{M}_{ \mathcal{P}_{n}}$ is probabilistic in essence and should be compared with the analytical proof of Chmutov \& Pittel \cite[Theorem 4.1]{chmutov2016surface} in the case when all perimeters are larger than $3$. The strategy followed by their proof is closer to that of Proposition \ref{prop:count_connected_maps}, with more involved calculations.

\subsection{Number of vertices and Theorem \ref{thm:vertices}}
We now turn to the proof of Theorem \ref{thm:vertices}, which is obviously based on \eqref{eq:Xtau}. We shall estimate separately the two contributions of \eqref{eq:Xtau} and start by controlling $ X_{ \tau_{n}}$, the number of vertices created during the exploration of the surface $ \mathbf{M}_{ \mathcal{P}_{n}}$ until time $\tau_{n}$:
\begin{lemma} \label{lem:Xtau}  If $( \mathcal{P}_{n})_{n \geq 1}$ is a good sequence of configurations, then
$$  \mathrm{d_{TV}}\left(X_{\tau_{n}}, \mathrm{Poisson}\left(  \log \frac{n}{n - \# \mathcal{P}_{n}}\right)\right) \xrightarrow[n\to\infty]{}0.$$
\end{lemma}
\begin{remark} \label{rem:tight} Notice that by \eqref{eq:2/3} the parameter in the Poisson law above asymptotically belongs to $[0, \log 3]$, which implies that $(X_{\tau_{n}})_{n \geq 1}$ is tight. This is the only fact we shall use to prove Theorem \ref{thm:vertices}.
\end{remark}
\proof Fix $n \geq1$ and perform the exploration of the surface $ \mathbf{M}_{ \mathcal{P}_{n}}$ using algorithm $ \mathcal{H}$ and stop at $\tau_{n}$. Using the description of the possible topology changes (Section \ref{sec:topologychanges}), for each $0 \leq i < n$, conditionally on $\mathcal{F}_i$, the number of vertices created by the next peeling step is 
$$ \left\{ \begin{array}{rcl}
1 &\mbox{ with probability }& \frac{2}{2(n-i)-1} \ind_{ \pi(i) \geq 3},\\
2 &\mbox{ with probability }& \frac{1}{2(n-i)-1} \ind_{ \pi(i)=2},\\
1 &\mbox{ with probability }& \frac{L_{i}-1}{2(n-i)-1} \ind_{ \pi(i)=1}.\\
\end{array}\right.$$
The proof of Proposition \ref{prop:connected} shows that the expected number of vertices created before time $3n/4$ by the last two possibilities is negligible as $n \to \infty$ and $\tau_{n} \leq 3n/4$ with high probability. Furthermore, by Proposition \ref{prop:loopbigon}, all but $o(n)$ peeling step take place on $p$-gons with $p \geq 3$. Since the sum of $o(n)$ Bernoulli variables with parameter bounded by $\frac{3}{n}$ is $0$ with high probability, we deduce that the law of $X_{\tau_{n}}$ is well approximated by 
$$ X_{\tau_{n}} \overset{ \mathrm{d_{TV}}}{\approx} \sum_{i=0}^{\tau_{n}} \mathrm{Bernoulli}\left( \frac{2}{2(n-i)-1}\right)\underset{ \mathrm{Prop. ~} \ref{prop:connected} \mathrm{\ and \ } \eqref{eq:taunborneinf}}{\overset{ \mathrm{d_{TV}}}{\approx}} \sum_{i=0}^{\# \mathcal{P}_{n}} \mathrm{Bernoulli}\left( \frac{1}{n-i}\right) \overset{\mathrm{d_{TV}}}{\approx} \mathrm{Poisson}\left(  \log \frac{n}{n - \# \mathcal{P}_{n}}\right).$$\qed 

Once $X_{\tau_{n}}$ is controlled, we need to get our hands on the other part of \eqref{eq:Xtau}, namely $ \mathbf{V}_{\{2n\}}$, which is the number of vertices of a unicellular map with $n$ edges. We prove the analog of our target result Theorem \ref{thm:vertices} for those random maps:
\begin{lemma} \label{lem:verticesunicellular}Let $ \epsilon_{n} \in \{ \mathrm{odd}, \mathrm{even}\}$ be the opposite parity of $n$. Then we have 
$$ \mathrm{d_{TV}}\left( \mathbf{V}_{\{2n\}}, \mathrm{Poisson}_{\log n}^{  \epsilon_{n}} \right) \xrightarrow[n\to\infty]{} 0.$$
\end{lemma}
\proof \textit{First proof.} This result is a straightforward consequence of the general theorem of Chmutov \& Pittel \cite{chmutov2016surface}, once we recalled that the number of cycles of a random permutation of $ \mathfrak{S}_{n}$ is close in total variation to a Poisson random variable of parameter $\log n$.\\
\textit{Second proof.} The well-known Harer--Zagier formula \cite{HZ86} precisely gives access to the generating function of the number of vertices of unicellular maps. This formula has been exploited to give a local limit theorem for $ \mathbf{V}_{\{2n\}}$ in \cite[Theorem 3.1]{chmutov2013genus}. Combining this local limit theorem with the explicit distribution of a Poisson random variable of parameter $\log n$ yields the result (we leave the straightforward calculations to the courageous reader).\\
\textit{Sketch of a third proof.} The above lines may disappoint the reader who expected that our results are ``self-contained'' and do not rely on any algebraic method. Let us explain how to get the lemma without relying on \cite{chmutov2016surface} nor \cite{HZ86}. The idea is to directly show that the number of vertices of a uniform unicellular map with $n$ edges is close (in total variation distance) to the number of vertices of a uniform map with $n$ edges, provided that it has the same parity as $n+1$. To fix ideas, let us suppose that $n$ is even. Let us thus consider $ \mathbb{M}_{n}^{  \mathrm{odd}}$, or by Theorem \ref{thm:maps==PDgluing}, up to a negligible error in total variation distance, $\mathbf{M}_{ \mathcal{U}_{2n}^{\mathrm{odd}}}$ where $ \mathcal{U}_{2n}^{ \mathrm{odd}}$ is the random variable $ \mathcal{U}_{2n}$ conditioned on $\# \mathcal{U}_{2n}$ being odd. We have
 \begin{eqnarray} \# \mathsf{V}\left(\mathbb{M}_{n}^{  \mathrm{odd}}\right)   &\underset{ \mathrm{duality}}{=}  &
 \# \mathsf{F}\left(\mathbb{M}_{n}\right) \mathrm{\ cond.~on~being~odd} \nonumber \\ &\underset{ \mathrm{Thm} \ref{thm:maps==PDgluing}}{\overset{ \mathrm{d_{TV}}}{\approx}}& \# \mathcal{U}_{2n}\mathrm{\ cond.~on~being~odd}  \nonumber \\
 &\underset{ \eqref{eq:countpoisson}}{\overset{ \mathrm{d_{TV}}}{\approx}} &\mathrm{Poisson}^{ \mathrm{odd}}_{ \log n}. \label{eq:verticesodd} \end{eqnarray}
On the other hand, let us  explore $ \mathbf{M}_{ \mathcal{U}_{2n}^{ \mathrm{odd}}}$ using algorithm $ \mathcal{H}$ until time $\tau_{n}$. Since the number of polygons of $ \mathcal{U}_{2n}^{ \mathrm{odd}}$ is typically of order $\log n$, we get from Proposition \ref{prop:connected} and \eqref{eq:taunborneinf} that $\tau_{n} \approx \log n$ and from Lemma \ref{lem:Xtau} that with high probability, no vertex has been created by the exploration until time $\tau_{n}$. Hence, we can write
$$  \mathrm{Poisson}^{ \mathrm{odd}}_{ \log n} \quad \overset{ \mathrm{d_{TV}}}{\underset{\eqref{eq:verticesodd}}{\approx}} \quad   \mathbf{V}_{\mathcal{U}_{2n}^{ \mathrm{odd}}} \quad  \overset{ \mathrm{d_{TV}}}{\underset{\begin{subarray}{c} \eqref{eq:Xtau}\\ \mathrm{Lem.} \ref{lem:tau}\\ \mathrm{Lem.} \ref{lem:Xtau} \end{subarray}}{\approx}} \quad \mathbf{V}_{\{2(n-\tau_{n})\}}.$$
We are almost there. Notice first that, with high probability, since we have not created vertices up to time $\tau_{n}$, then $\tau_{n}$ is even and so is $n-\tau_{n}$. The above line shows that the number of vertices of a unicellular map with a \emph{random} number of edges $n-\tau_{n}$ is close in total variation to our goal $ \mathrm{Poisson}^{ \mathrm{odd}}_{\log n}$. To finish the proof, it remains to see that if $n'>n$ have the same parity and $ n'-n = O( \log n)$, then $ \mathbf{V}_{\{2n\}} \approx \mathbf{V}_{ \{2n'\}}$ in total variation distance. To see this, we will couple the two discrete surfaces $ \mathbf{M}_{\{2n\}}$ and $ \mathbf{M}_{\{2n'\}}$ so that they have the same number of vertices with high probability. The idea will be to couple their explorations using algorithm $ \mathcal{H}$ in such a way that they are independent until some stopping time $\xi$, and coincide afterwards. The key is that during these explorations, the number of holes is a process which spends most of its time on small values. More precisely, let us explore independently $ \mathbf{M}_{ \{2n\}}$ and $ \mathbf{M}_{\{2n'\}}$ using algorithm $ \mathcal{H}$ and denote by $H$ and $H'$ the processes of the number of holes in each exploration. Notice that as long as we do not create vertices (which happens only after $ \approx n$ steps by Lemma \ref{lem:Xtau}), the numbers $H_{i}$ and $H'_{i}$ have the same parity. We will be interested in the stopping time
\[\xi = \inf \{i \geq n'-n : H'_{n'-n+i} = H_{i} =1 \}.\]
By \eqref{eq:driftH} (see also the proof of Lemma \ref{lem:tau}), the processes $(H_i)$ and $(H'_{n'-n+i})$ are dominated (up to time $ \approx n$) by two independent copies $X$ and $X'$ of a positive recurrent process, where $X$ and $X'$ have the same parity. In particular, there is a small $i$ such that $X_i=X'_i=1$. It follows that $\xi$ happens quickly in the sense that 
$$ \mathbb{P}( \xi \geq C\log n \mbox{ and  no vertex has been created by then in either surfaces}) \xrightarrow[n \to \infty]{} 1.$$
Therefore, the surfaces $S_{\xi}$ and $S'_{n'-n+\xi}$ both have exactly one hole, with the same size $2(n-\xi)$. Hence, on the event described above, we can couple $ \mathbf{M}_{ \{2n\}}$ and $ \mathbf{M}_{\{2n'\}}$ by identifying their explorations from time $\xi$ and $\xi + n'-n$ respectively on. In particular, when this coupling occurs, we have $\mathbf{V}_{\{2n\}} = \mathbf{V}_{ \{2n'\}}$ as desired. We leave the details to the interested reader. \qed 

\proof[Proof of Theorem \ref{thm:vertices}] Recalling \eqref{eq:Xtau} we write $\mathsf{V}_{ \mathcal{P}_{n}} = X_{\tau_{n}} + \mathsf{V}_{\{2(n-\tau_{n})\}}$ where conditionally on $\tau_{n}$ the variables $ \mathbf{V}_{\{2(n-\tau_{n})\}}$ and $X_{\tau_{n}}$ are independent. Hence 
 \begin{eqnarray*}
 \mathsf{V}_{ \mathcal{P}_{n}} &=& X_{\tau_{n}} + \mathsf{V}_{\{2(n-\tau_{n})\}}\\
 & \underset{  \begin{subarray}{c}  \mathrm{Lem.} \ref{lem:verticesunicellular}\end{subarray}}{ \overset{ \mathrm{d_{TV}}}{\approx}} & X_{\tau_{n}} + \mathrm{Poisson}_{\log (n- \tau_{n})}^{ \mathrm{parity}( n-\tau_{n}+1)}\\ 
  & \underset{\begin{subarray}{c} \mathrm{Lem.}\ref{lem:tau} \end{subarray}}{ \overset{ \mathrm{d_{TV}}}{\approx}} &  X_{\tau_{n}} + \mathrm{Poisson}_{\log n}^{ \mathrm{parity}(n-\tau_{n}+1)}\\ 
 & \underset{\begin{subarray}{c} \mathrm{Lem.} \ref{lem:Xtau} \end{subarray}}{ \overset{ \mathrm{d_{TV}}}{\approx}} &  \mathrm{Poisson}_{\log n}^{ \mathrm{parity}( X_{\tau_{n}}+n-\tau_{n}+1)}.  \end{eqnarray*}
The second-to-last equality uses the fact that $\mathrm{Poisson}_{\log n}$ and $\mathrm{Poisson}_{\log (n-i)}$ are close for $ \mathrm{d_{TV}}$ uniformly over $0 \leq i \leq \frac{3n}{4}$. The last one uses the fact that $X_{\tau_n}$ is tight and for any fixed $i$, the variables $\mathrm{Poisson}_{\log n}$ and $i+\mathrm{Poisson}_{\log n}$ are close for $ \mathrm{d_{TV}}$ as $n \to +\infty$. Finally, if $ \mathbf{M}_{ \mathcal{P}_{n}}$ is connected, by applying the Euler formula to $S_{\tau_n}$, it is easy to check that $ X_{\tau_{n}}+n-\tau_{n}+1$ has the same parity as $ n + \# \mathcal{P}_{n}$.  We just landed on the desired Poisson variable having the correct parity. \endproof

\section{Peeling vertices and the Poisson--Dirichlet universality}
\label{subsec:PD_universal}
In this section we prove the Poisson--Dirichlet universality, that is our Theorem \ref{thm:weakPD}. Again, the idea is to explore our random surface via a peeling algorithm tailored to our objective. Since we are interested in the vertex degrees, our algorithm will explore the $1$-neighborhood of a given vertex. Once all the edges adjacent to this vertex have been discovered, we choose a new vertex on the boundary of the current surface and iterate. More precisely: \medskip 

\fbox{ \begin{minipage}{14cm}
\textbf{Algorithm $\mathcal{R}$ or peeling vertices}: Given the initial configuration of  labeled polygons $S_{0}\equiv \mathcal{P}$ we pick a ``red'' vertex $ \mathfrak{R}_{0}\in \partial S_{0}$ uniformly at random. Inductively, given the discrete surface $S_{i}$ with a distinguished ``red'' point  $ \mathfrak{R}_{i} \in  \partial S_{i}$, we peel the edge lying immediately on the left of $ \mathfrak{R}_{i}$ to get $S_{i+1}$. If during this peeling step the red vertex has been swallowed by the process, then we resample $ \mathfrak{R}_{i+1} \in   \partial S_{i+1}$ uniformly at random (independently of the past and of the gluing). Otherwise $ \mathfrak{R}_{i+1}$ canonically results of $ \mathfrak{R}_{i}$.
\end{minipage}}

\medskip

It is easy to see that the above algorithm is again Markovian and so we can apply Proposition \ref{prop:markov}. We shall always use the above algorithm when exploring our surfaces in this section.

\subsection{Closure times and targeting distinguished vertices}

If $S_{i}$ is a discrete surface with a red vertex $ \mathfrak{R}_{i} \in \partial S_{i}$, then the peeling of the edge immediately on the left of $\mathfrak{R}_{i}$ leaves the red vertex on the boundary of $S_{i+1}$ except in two cases: 
\begin{itemize}
\item if the peeled edge is glued to the edge immediately on the right of $\mathfrak{R}_{i}$, see Figure \ref{fig:case4}. We say that time $i$ is then a \emph{strong closure time};
\item or if the red vertex $\mathfrak{R}_{i}$ belongs to a  hole of perimeter $1$, which is glued to another hole of perimeter $1$, see Figure~\ref{fig:case5} left. In this case, we say that $i$ is a \emph{loop closure time}.
\end{itemize}
Notice that in both cases the peeling step yields the creation of at least $1$ true vertex of $ \mathbf{M}_{ \mathcal{P}}$ (this can be $2$ if the peeling step closes the two sides of a bigon). There is actually another scenario which yields to the creation of a true vertex at time $i$: if the peeled edge is identified to the edge immediately on its left. We call these times \emph{weak closure times} and we will see in the proof that the vertices created there have a low degree and so do not affect the Poisson--Dirichlet universality. We will denote by  $$0 = \theta^{(0)} \leq \theta^{(1)} < \theta^{(2)} < \cdots$$ the strong closure times with the convention that $\theta^{(0)}=0$. If $ |\mathcal{P}|=n$, then we have 
 \begin{eqnarray} \label{eq:sctvrai}\mathbb{P}( i \mbox{ is a strong closure time} \mid \mathcal{F}_{i}) = \frac{1}{2(n-i)-1} \ind_{\mathfrak{R}_{i} \mbox{ is not on a loop }}.  \end{eqnarray}
In the rest of this section, we will be interested in asymptotic properties as $| \mathcal{P}| \to \infty$ so as expected, we fix a good sequence of configurations $( \mathcal{P}_{n})_{n \geq 1}$ and explore $ \mathbf{M}_{ \mathcal{P}_{n}}$ using the peeling algorithm $ \mathcal{R}$. To highlight the dependence in $n$ we write $\theta_{n}^{(i)}$ for the strong closure times. To describe the degree of the vertices created by strong closure times and their connectivity it will be convenient to work with ``finite dimensional marginals''. Specifically, imagine that we number arbitrarily the vertices on $S_{0}$ from $1$ up to $2n$ and keep track of these labels during the exploration procedure. More precisely, these labels are ``merged'' if two vertices coalesce and they can disappear from $ \partial S_{i}$ if they are swallowed during a strong, weak or loop closure time. We denote by $\sigma_{n}^{(j)}$ the time at which the vertex carrying the label $j$ is identified with the red vertex. If the label  $j$ disappears before being glued to the red vertex, then we write $\sigma_{n}^{(j)}= \infty$ by convention.
\begin{proposition}[Closure times and Poisson--Dirichlet of parameter $1/2$] \label{prop:closure} Let $( \mathcal{P}_{n})_{n\geq 1}$ be a good sequence of configurations. We have the convergence in distribution in the sense of finite dimensional marginals
$$\left( \left( \frac{\theta_{n}^{(i)}-\theta_n^{(i-1)}}{n-\theta_{n}^{(i-1)}}\right), \left( \frac{\sigma_{n}^{(i)}}{n}\right) \right)_{i \geq 1} \xrightarrow[n\to\infty]{}   
\big(T_{i}, S_{i}\big)_{i \geq 1},$$
where $(T_i)_{i \geq 1}$ and $ (S_{i})_{i \geq 1}$ are  i.i.d. random variables with distribution $\beta \left( 1, \frac{1}{2} \right)$, i.e. variables with density $ \displaystyle \frac{1}{2 \sqrt{1-t}} \mathbbm{1}_{t \in [0,1]}$ with respect to the Lebesgue measure.  \end{proposition}

\proof[Preparation to the proof.] To prepare the reader to the proof, we first concentrate on an approximate model for which the result is easy to get. We assume we are in an ideal situation, i.e. that there are no loop closure times and that for each time $0 \leq i \leq n$, the red vertex $ \mathfrak{R}_{i}$ has conditional probability exactly 
$$ \frac{1}{2(n-i)-1}$$ (to be compared to \eqref{eq:sctvrai}) to be swallowed in a strong closure time. We furthermore assume that the $k$ labels $1,2,3,\dots,k$ are always carried by distinct vertices and each of them has probability exactly $\frac{1}{2(n-i)-1}$ of being glued to the red vertex (and thus to disappear) at time $i$, provided it has not been glued to the red vertex so far. We also suppose that the above events (strong closure time, or gluing of one of the remaining labels) are mutually exclusive (to be precise, one should stop the process when $k+1 > 2(n-i)-1$). We denote by $\tilde{\theta}^{(i)}_{n}$ and $\tilde{\sigma}_{n}^{(1)}, \dots , \tilde{\sigma}_{n}^{(k)}$ the associated strong closure times and gluing times.
We first prove the convergence of $\left( \frac{\tilde{\theta}_{n}^{(i)}}{n} \right)_{i \geq 1}$. For every $x \in (0,1)$, we have
\begin{equation}\label{eqn:limit_beta}
\mathbb{P}(\tilde{\theta}^{(1)}_n \geq xn) = \prod_{i=0}^{[xn]} \left(1- \frac{1}{2(n-i)-1}\right) \xrightarrow[n \to \infty]{} \sqrt{1-x},
\end{equation}
which gives the convergence of $\frac{\tilde{\theta}^{(1)}_n}{n}$ to $T_1$. Now conditionally on $\tilde{\theta}^{(1)}_n=xn$, the process $( \tilde{\theta}^{(i+1)}_n )_{i \geq 1}$ has the same distribution as $( xn + \tilde{\theta}^{(i)}_{(1-x)n} )_{i \geq 1}$ 
and by an easy induction, we obtain the convergence of the first term in Proposition 23. Moreover, the computation \eqref{eqn:limit_beta} also shows, for every $i$, the convergence of $n^{-1} \cdot \tilde{\sigma}_n^{(i)}$ to a variable with distribution $\beta ( 1,\frac{1}{2} )$. Therefore, it remains to prove that in the limit, the variables $\tilde{\sigma}_n^{(i)}$ are independent of each other and of $( \tilde{\theta}_n^{(i)} )_{i \geq 1}$. For this, it is easy to see that conditionally on $( \tilde{\theta}_n^{(i)} )_{i \geq 1}$, they have the same distribution as i.i.d. variables $\tilde{\tilde{\sigma}}_n^{(i)}$ conditionned to be pairwise distinct and different from the $\tilde{\theta}_n^{(i)}$. On the other hand, for any fixed $i \ne j$, the probability that $\tilde{\tilde{\sigma}}_n^{(i)}=\tilde{\tilde{\sigma}}_n^{(j)}$ or $\tilde{\tilde{\sigma}}_n^{(i)}=\tilde{\theta}_n^{(j)}$ goes to $0$ as $n \to +\infty$. Therefore, this conditioning has negligible effect in total variation distance, which proves the claim. \qed
\bigskip

Compared to the above ideal situation, the true exploration has the following differences:
\begin{itemize}
\item strong closure times are ruled by \eqref{eq:sctvrai} and thus are perturbed when $ \mathfrak{R}_{i}$ is on a loop,
\item some labels $1,2,3,\dots$ could be merged together before being glued to the red vertex,
\item some labels $1,2,3,\dots$ could even be swallowed by the process (by a weak closure time) before being glued to the red vertex,
\item when resampling $ \mathfrak{R}_{i}$ after a closure time, we may identify a given label to the red vertex. 
\end{itemize}
Those annoying situations will be ruled out with high probability in the next section thus reducing our exploration model to the above ideal situation. Before doing that, let us show how Proposition \ref{prop:closure} implies our Theorem \ref{thm:weakPD}.

\paragraph{Proof of Theorem \ref{thm:weakPD}.}
With the notation of Proposition \ref{prop:closure}, we denote by $v_n^{(i)}$ the vertex of $\mathbf{M}_{\mathcal{P}_n}$ created by the $i$-th strong closure time $\theta_n^{(i)}$. Let $j$ be a fixed label, and consider the unique $i$ such that 
\begin{equation}
\theta_n^{(i-1)} \leq \sigma_n^{(j)} < \theta_n^{(i)}.
\end{equation}
As explained above, the possibility that the vertex label $j$ is swallowed by a weak or a loop closure time will be ruled out later (Lemma \ref{lem:weakclosing}). Therefore, the label $j$ will eventually land on the vertex $v_i^{(n)}$ in the end of the exploration.

The idea of the proof will be to pick $k$ edges uniformly at random on the polygons of $\mathcal{P}_n$ and to label their endpoints as $1,2,\dots,2k$. After the gluing, these edges result in $k$ uniform edges of $\mathbf{G}_{ \mathcal{P}_{n}}$. Note that since we are working with a good sequence of configurations, the probability for one of these $k$ edges to be a loop is $o(\sqrt{n})$, so with high probability the labeled vertices are pairwise distinct before the exploration. Given the graph $ \mathbf{G}_{ \mathcal{P}_{n}}$, let $ \mathtt{G}_{k,n}$ be the multi-graph induced by these $k$ uniform edges by throwing out all isolated points (this graph may be disconnected). Then the graph $ \mathtt{G}_{k,n}$ is completely described by the relative order of the numbers $\sigma_n^{(j)}$ and $\theta_n^{(i)}$, which is itself described by Proposition \ref{prop:closure}. Since relative order does not change if we apply to these numbers an increasing function, one can apply the map
$$ \phi : x \in[0,1] \mapsto \sqrt{1-x} \in [0,1]$$
to the numbers $n^{-1}\sigma_{n}^{(j)}$ and $ n^{-1}\theta^{(i)}_{n}$. Then Proposition \ref{prop:closure} and an easy calculation show that they converge respectively towards $U_{1},U_{2}, \dots$ independent uniform random variables on $[0,1]$ and towards $V_{1}, V_{2}(1-V_{1}), V_{3}(1-V_{1})(1-V_{2}), \ldots$ where $(V_{i} : i \geq 1)$ are independent uniform random variables on $[0,1]$. In other words, $(n^{-1}\theta_n^{(i)})_{i \geq 1}$ converges towards the stick breaking construction of the Poisson--Dirichlet partition $(X_i)_{i \geq 1}$. Hence, the limit of $\mathtt{G}_{k,n}$ as $n \to \infty$ can be described by the variables $(U_j)$ and the Poisson--Dirichlet partition $(X_i)$. More precisely, consider the following random graph: start with a Poisson--Dirichlet partition $(X_{i} : i \geq 1)$ of $[0,1]$ and throw $k$ independent pairs $(U_{j}^{(1)}, U_{j}^{(2)})$ of independent uniform random variables on $[0,1]$. If $U_{j}^{(1)}$ falls in the interval $X_{a}$ and $U_{j}^{(2)}$ in $X_{b}$, then add an edge between the vertices $a$ and $b$ (this may be a loop). Denote the induced subgraph (by throwing out the isolated vertices) by $ \mathtt{G}_{k,\infty}$. This is a particular case of the construction \cite[Example 7.1]{janson2018edge}.

Since an oriented edge of $ \mathbf{M}_{ \mathcal{P}_{n}}$ is obtained by picking two consecutive vertices on $S_{0}$, Proposition \ref{prop:closure} implies that
\begin{equation}\label{eqn:convergence_k_edges}
\forall k \geq 1, \quad \mathtt{G}_{k,n} \xrightarrow[n\to\infty]{(d)} \mathtt{G}_{k,\infty}.
\end{equation}
The end of the proof is now quite standard: given $\mathbf{G}_{\mathcal{P}_n}$, the number $[i,j]_{\mathtt{G}_{k,n}}$ of edges of $\mathtt{G}_{k,n}$ joining the vertices $v_i^{(n)}$ and $v_j^{(n)}$ is a binomial variable $\mathrm{Bin} \left( k, n^{-1} [i,j]_n \right)$. Therefore, for $k$ large, uniformly in $n$, the variable $k^{-1}[i,j]_{\mathtt{G}_{k,n}}$ is concentrated around $n^{-1} [i,j]_n$. On the other hand, by \eqref{eqn:convergence_k_edges}, we have
\[ k^{-1} [i,j]_{\mathtt{G}_{k,n}} \xrightarrow[n \to +\infty]{} k^{-1} \mathrm{Bin} \left( k, X_i X_j  \right) \xrightarrow[k \to +\infty]{} X_i X_j,\] in distribution.
 The convergence $[i,j]_n \to X_i X_j$ when $n \to +\infty$ follows.

\begin{remark}[Hypergraph extension.]\label{rem:hypergraph} The above scheme of proof can be pushed further when our configuration has a positive density of triangles, quadrangles, etc. For example, let us consider the case when $ \mathcal{P}_{n}$ is made of $n/3$ triangles only\footnote{supposing that $n$ is divisible by $3$} as in \cite{BM04}. With the same notation as in the introduction, if we denote by $[i,j,k]_{ \mathbf{G}_{ \mathcal{P}_{n}}}$ the number of triangles whose endpoints are the $i$-th, $j$-th and $k$-th largest degree vertices, then we have
$$ \left(\frac{[i,j,k]_{n}}{n/3} : i,j,k \geq 1\right) \xrightarrow[n\to\infty]{(d)} \left( X_{i}\cdot X_{j} \cdot X_{k} : i,j,k \geq 1\right).$$
From this convergence, one can deduce limit laws for the number of triangles incident to the origin vertex, for the total number of  ``folded'' triangles i.e. whose three apexes are confounded and so on. 
\end{remark}

\subsection{A few bounds}
In this section we rule out the bad situations described after Proposition \ref{prop:closure} to reduce its proof to the ideal situation already considered.  We start with controlling loop closure times and the time spent on peeling loops.
\begin{lemma} \label{lemma:peuloops} If $( \mathcal{P}_{n})_{n \geq 1}$ is a good sequence of configurations and if $ \mathtt{Lct}_{n}$ is the first loop closure time, with the convention that $ \mathtt{Lct}_{n} =n$ if there are no such times, then
$$ \frac{ \mathtt{Lct}_{n}}{n} \xrightarrow[n\to\infty]{( \mathbb{P})} 1 \quad \mbox{and} \quad \frac{1}{n}  \sum_{i=0}^{ n-1} \ind_{ \mathfrak{R}_{i} \mbox{ is on a loop}}  \xrightarrow[n\to\infty]{( \mathbb{P})} 0.$$
\end{lemma}
\proof The second convergence is a straightforward consequence of the fourth convergence of Proposition \ref{prop:loopbigon}. To prove the first one, conditionally on the past exploration up to time $i$, the probability to perform a loop closure time is equal to $ \ind_{\pi(i)=1}  \frac{L_{i}-1}{2(n-i)-1}$. Hence, summing over $ 0 \leq i \leq (1- \varepsilon)n$, taking expectation and splitting according to whether $\sup L_{i} \geq \varepsilon \sqrt{n}$, we get for every $\eps>0$:
 \begin{eqnarray*} \mathbb{P}( \mathtt{Lct}_{n} \leq (1- \varepsilon)n) &\leq&  \mathbb{E}\left[\sum_{i=0}^{[(1- \varepsilon)n]}  \ind_{\pi(i)=1} \frac{\sup_{0 \leq j \leq n} L_{j}}{2(n-i)-1}\right] \\
  & \leq & \frac{1}{\sqrt{n}} \cdot \mathbb{E}\left[\sum_{i=0}^{[(1- \varepsilon)n]}  \ind_{\pi(i)=1} \right] + n \mathbb{P} \left( \sup_{0 \leq j \leq n} L_{j} \geq \varepsilon \sqrt{n} \right).  \end{eqnarray*}
By our goodness assumption and Proposition \ref{prop:loopbigon}, this tends to $0$.
\endproof

We will need to rule out a few other annoying situations. Assume we track a distinguished label during the exploration, say the  label $1$. Note that the only case where this label is never glued to the red vertex is if it is swallowed at some point by a weak closure time. Part of Proposition \ref{prop:closure} is that this situation does not occur (this is important since it ensures that the edges of $ \mathbf{M}_{ \mathcal{P}_{n}}$ are concentrated on the vertices closed at strong closure times). We will also prove that with high probability, two fixed distinguished labels do not coalesce before being glued to the red vertex (this is important to rule out strong correlations between the times $\sigma_{n}^{(j)}$).
\begin{lemma}\label{lem:weakclosing}
Assume that $( \mathcal{P}_{n})_{n \geq 1}$ is a sequence of good configurations and perform the exploration of $ \mathbf{M}_{ \mathcal{P}_{n}}$ using algorithm $ \mathcal{R}$ after having labeled the vertices of $ S_{0}$ by $\{1,2, \dots , 2n\}$ arbitrarily (independently of the matching of the edges). Then for every $ \varepsilon>0$, with high probability as $n \to \infty$, none of the following events occur before time $(1- \varepsilon)n$:
\begin{enumerate}
\item the label  $1$ disappears before being glued to the red vertex; \label{lem_weak_closing_negligible} 
\item the labels $1$ and $2$ coalesce before being glued to the red vertex; \label{lem_no_correlation_s}
\item the red vertex is moved to the vertex  carrying the label $1$ after some strong closure time. 
\end{enumerate}
\end{lemma}

Of course, once the lemma is proved for labels $1$ and $2$, it easily extends to the labels $1,2,\dots, k$ for any fixed $k$.

\proof We start with the first item. Fix $ \varepsilon>0$. In the event $A_{i}$ where the label $1$ is swallowed at time $i$ by a weak closure time, this label is necessarily carried by the vertex immediately to the left of $ \mathfrak{R}_{i}$ and a weak closure time happens at time $i$. Hence 
  \begin{eqnarray} \label{eq:1bouffe} \mathbb{P}( A_{i}) \leq \frac{1}{2(n-i)-1} \mathbb{P}( 1 \mbox{ is carried by the vertex on the left of } \mathfrak{R}_{i}).  \end{eqnarray}
We write $\alpha_{i} = \mathbb{P}( 1 \mbox{ is carried by the vertex on the left of } \mathfrak{R}_{i})$.
Then we can estimate $\alpha_{i}$ for $i \geq 1$ by looking at the peeling step $i-1$ as follows:
\begin{itemize}
\item Either the $(i-1)$-th peeling step swallows the red vertex and the new one is sampled uniformly on the boundary. In this case, the probability that $ \mathfrak{R}_{i}$ is on the right of the label $1$ is 
$$ \frac{1}{2(n-i)}.$$
\item Either the $(i-1)$-th peeling step glues a bigon on the left of $ \mathfrak{R}_{i-1}$ and, if the label $1$ was already on the left of $\mathfrak{R}_{i-1}$, it stays on the left of $ \mathfrak{R}_{i}$. The conditional probability of this scenario is thus bounded above by 
$$ \frac{2 B_{i-1}}{2(n-i)+1} \mathbbm{1}_{1 \mbox{ is carried by the vertex on the left of } \mathfrak{R}_{i-1}} .$$
\item Or the $(i-1)$-th peeling step glues a loop on the left of $ \mathfrak{R}_{i-1}$ and if the label was on the second vertex on the left of $ \mathfrak{R}_{i-1}$, then it becomes immediately on the left of $ \mathfrak{R}_{i}$. We can crudely bound the conditional probability of this event by 
$$  \frac{L_{i}}{ 2(n-i)+1}.$$
\item In all other situations, in order for the label $1$ to be on the left of $ \mathfrak{R}_{i}$, the $(i-1)$-th peeling step should identify the peeled edge with the second edge on the right of the label $1$, which occurs with probability 
$$ \frac{1}{2(n-i)+1}.$$
\end{itemize}
In total, for any $i \leq (1- \varepsilon)n$, taking expectation and then splitting according to whether $ \sup_{0 \leq j \leq n} L_{j} \geq \varepsilon^2  \sqrt{n}$ and $\sup_{0 \leq j \leq n} B_{j} \geq \varepsilon^2  n$, we have for large $n$ and $0 \leq i \leq (1-\eps)n$
 \begin{eqnarray*} \alpha_{i} &\leq& \frac{2}{2(n-i)} +  \mathbb{E}\left[\frac{2 B_{i-1}}{2(n-i)+1} \mathbbm{1}_{1 \mbox{ is carried by the vertex on the left of } \mathfrak{R}_{i-1}}\right] + \mathbb{E}\left[\frac{L_{i}}{ 2(n-i)+1}\right] \\
 & \leq & \frac{C}{n} + 2 \varepsilon \cdot \alpha_{i-1} + \mathbb{P}( \sup B_{j} \geq \varepsilon^2 n)   + \frac{ \varepsilon}{ \sqrt{n}} + \mathbb{P}( \sup L_{j} \geq  \varepsilon^2 \sqrt{n}). \end{eqnarray*}
Using Proposition \ref{prop:loopbigon}, and by our goodness assumption, we see that the two probabilities in the right-hand side are negligible compared to $ \frac{ \varepsilon}{ \sqrt{n}}$ for large $n$ and so we get 
$$ \alpha_{i} \leq \frac{2 \varepsilon}{ \sqrt{n}} + 2 \varepsilon \cdot \alpha_{i-1}.$$ When $2 \varepsilon <1$ this easily implies that $\alpha_{i} \leq C_{\eps}/ \sqrt{n}$ uniformly in $i \leq (1- \varepsilon)n$ as $n \to \infty$ for some constant $C>0$ depending on $ \varepsilon$. Plugging this back in \eqref{eq:1bouffe}, we obtain $\PP(A_{i}) \leq \frac{C_{\eps}}{n^{3/2}}$ for $i \leq (1-\eps)n$, so $\PP \left( \bigcup_{i=0}^{(1-\eps)n} A_i \right)$ goes to $0$ as $n \to \infty$. 

For the second item, if we want the labels $1$ and $2$ to merge in the same vertex, then one of the two, say $1$, must be immediately on the left of $ \mathfrak{R}_{i}$ at time $i$ and then the peeling step should identify the edge on the left of $ \mathfrak{R}_{i}$ with the edge on the left of the label $2$. This has probability $\alpha_{i} \frac{1}{2(n-i)-1}$ and the last calculation shows that after summing over $i \leq (1- \varepsilon)n$ we get a negligible contribution.

The third item is the most obvious:  the probability that $i$ is a strong closure time and that the red vertex is moved to the label $1$ is at most 
$$ \frac{1}{2(n-i)-1} \times \frac{1}{2(n-i)-3},$$
and the result follows by summing over $i \leq (1- \varepsilon)n$. \qed

\bigskip 

\noindent \textsc{Thomas Budzinski,\\ DMA, ENS Paris, 75005 Paris, France}\\ \\
\noindent \textsc{Nicolas Curien,\\ Laboratoire de Math\'ematiques d'Orsay, Univ. Paris-Sud, CNRS, Universit\'e Paris-Saclay, 91405 Orsay, France}\\ \\
\noindent \textsc{Bram Petri,\\ Mathematisches Institut der Universit\"at Bonn, 53115 Bonn, Germany}

\bibliographystyle{siam}
\bibliography{bibli-total}

\end{document}